\documentclass[11pt]{amsart}

\usepackage{euscript}
\usepackage{amsmath}
\usepackage{amsfonts}
\usepackage{amssymb}
\usepackage{mathrsfs}
\usepackage{amsthm}
\usepackage{epsfig}
\usepackage{epstopdf}
\usepackage{url}
\usepackage{array}
\usepackage{amssymb}\usepackage{amscd}
\usepackage{epic}
\usepackage{tikz,underscore}
\usepackage{tikz-cd}
\usepackage{float}
\usepackage{subfigure}
\usepackage{tkz-euclide}
\usepackage{yhmath}

\usepackage{yfonts}

\usepackage{graphicx}

\numberwithin{equation}{section}

\theoremstyle{plain}
\newtheorem{theorem}{Theorem}[section]
\newtheorem{lemma}[theorem]{Lemma}
\newtheorem{proposition}[theorem]{Proposition}

\newtheorem{corollary}[theorem]{Corollary}

\newtheorem{claim}{Claim}
\newtheorem*{claim*}{Claim}

\newtheorem{ques}[theorem]{Question}

\theoremstyle{definition}

\newtheorem{remark}[theorem]{Remark}

\newtheorem{definition}[theorem]{Definition}

\DeclareMathOperator{\Isom}{{\mathrm Isom}}
\DeclareMathOperator{\Hull}{{\mathrm Hull}}
\DeclareMathOperator{\stab}{{\mathrm stab}}

\def\La{\Lambda}

\def\Ga{\Gamma}
\def\ga{\gamma}

\def\H{\mathbb H}

\def\R{\mathbb R}

\newcommand{\fix}{\textup{Fix}}

\def\geo{\partial_{\infty}}

\title{Discrete subgroups of small critical exponent}

\author{Beibei Liu}
\address{School of Mathematics, Georgia Institute of Technology, Atlanta, GA 30332, USA}
\email{bliu96@gatech.edu}

\author{Shi Wang}
\address{Department of mathematics, Michigan State University, East Lansing, MI 48824, USA}
\email{shiwang.math@gmail.com}


\providecommand{\keywords}[1]{\textbf{\textit{Index terms---}} #1}

\date{}

\begin{document}

\begin{abstract}

We prove that finitely generated Kleinian groups $\Gamma<\Isom(\H^{n})$ with small critical exponent are always convex-cocompact. Along the way, we also prove some geometric properties for any complete pinched negatively curved manifold with critical exponent less than 1.

\end{abstract}


\maketitle

\section{Introduction}
\label{sec:introduction}
A \emph{Kleinian} group is a discrete isometry subgroup of $\Isom(\H^{n})$. The study of $3$-dimensional finitely generated Kleinian groups dates back to Schottky, Poincar\'e and Klein. It is only recently that the geometric  picture of the associated hyperbolic manifold  has been  much better understood, after the celebrated work of Ahlfors' finiteness theorem \cite{Ah}, the proof of the tameness conjecture \cite{Bon, Agol, Gabai}, and the unraveling of the Ending Lamination Conjecture \cite{Min, BCM, som, bow2}. However,  such geometric descriptions  fail in higher dimensions \cite{Kap1, Kap2, KP1, KP2, Poty, poty2}. 

To study higher dimensional Kleinian groups, one way  is to consider  the interplay between the group theoretic properties, the geometry of the quotient manifolds, and the \emph{measure-theoretic size} of the limit set.   It is shown in \cite{gus} that if the Hausdorff dimension of the entire limit set  $\dim_{\mathcal H}(\Lambda(\Gamma))<1$, then $\Gamma$ is geometrically finite.  In such case, the Hausdorff dimension of the entire limit set equals the Hausdorff dimension of the conical limit set \cite{bow3} which is smaller than $1$. However, when $\Gamma$ is geometrically infinite, the size of the entire limit set could a priori be much larger so as $\dim_{\mathcal H}\Lambda(\Gamma)>\dim_{\mathcal H}\Lambda_c(\Gamma)$. Thus, it is interesting to ask what is the relative size of $\Lambda_c(\Gamma)$ compared to the entire $\Lambda(\Gamma)$, or rather, to what extent is the size of $\Lambda_c(\Gamma)$ able to determine the geometric finiteness of the group. By the work of Bishop and Jones \cite{BJ}, the Hausdorff dimension of the conical limit set $\Lambda_c(\Gamma)$ equals the critical exponent $\delta(\Gamma)$. Hence, Kapovich \cite[Problem 1.6]{Kap3} asked the following question:

\begin{ques}
Is every finitely generated Kleinian group $\Gamma<\operatorname{Isom}(\mathbb H^n)$ with $\delta(\Gamma)<1$ geometrically finite?
\end{ques}

In the present paper, we partly answer this in the affirmative while considering in a slightly more general context.

\begin{theorem}
\label{thm:main} For each $n$ and $\kappa$, there exists a positive constant $D(n, \kappa)<1/2$ with the following property that: for every $n$-dimensional Hadamard manifold with pinched sectional curvature $-\kappa^{2}\leq K\leq -1$ and any finitely generated, torsion-free, discrete isometry subgroup $\Gamma<\operatorname{Isom}(X)$, $\Gamma$ is convex cocompact if $\delta(\Gamma)<D(n,\kappa)$.
\end{theorem}

\begin{remark}
The constant $D(n, \kappa)$ can be obtained from the quantitative version of the Tits alternative for pinched negatively curved manifolds \cite{Dey-Kapovich-Liu}.
\end{remark}

\begin{remark}
For $3$-dimensional finitely generated  Kleinian groups $\Gamma$ of \emph{second kind}, i.e., $\La(\Ga)\neq S^{2}$, Bishop and Jones \cite{BJ} showed that $\Gamma$ is geometrically finite if $\delta(\Gamma)<2$.  Hou \cite{Hou1, Hou2, Hou3} proved that a 3-dimensional Kleinian group $\Ga$ is a classical Schottky group if $\dim_{\mathcal H}(\Lambda(\Gamma))<1$.
\end{remark}

In \cite{Kap3}, Kapovich established a relation between the homological dimension and the critical exponent of a Kleinian group. Similar homological vanishing feature has been extended to other rank one symmetric spaces by Connell, Farb and McReynolds \cite{CFM}. It is conjectured \cite[Conjecture 1.4]{Kap3} that the virtual cohomological dimension $\textup{vcd}(\Gamma)$ is bounded above by $\delta(\Gamma)+1$ (assume $\Ga$ has no higher rank cusps). Under the condition $\delta(\Gamma)<1$, it is equivalent \cite{Stal} to ask (See also a weaker form in \cite[Question 5.6]{Bestvina}.):

\begin{ques}\label{ques:free}
Is every finitely generated Kleinian group $\Gamma<\operatorname{Isom}(\mathbb H^n)$ with $\delta(\Gamma)<1$ virtually free?
\end{ques}

In the same paper, Kapovich gave a positive answer to this question under a stronger assumption that $\Gamma$ is finitely presented. On the other hand, when $\delta(\Gamma)$ is sufficiently small, our Theorem \ref{thm:main} automatically implies $\dim_{\mathcal H}(\Lambda(\Gamma))=\delta(\Gamma)<D(n, \kappa)<1$. This implies that the limit set $\La(\Ga)$ is a Cantor set since it is perfect. Following the classical result of  Kulkarni \cite[Theorem 6.11]{Kulkarni78},  we have:

\begin{corollary}
\label{coro:virtuallyfree}
For each $n$, there exists a positive constant $D(n)<1/2$ such that, any finitely generated discrete isometry subgroup $\Gamma<\Isom(\H^{n})$ is virtually free if $\delta(\Gamma)<D(n)$.
\end{corollary}

\begin{remark}
Under the assumption that $\dim_{\mathcal H}(\Lambda(\Gamma))<1$, Pankka and Souto \cite{PS} proved that  any torsion free Kleinian group (not necessarily finitely generated) is free. 

\end{remark}

The method in  \cite{Kap3} also works for discrete isometry subgroups of Hadamard manifolds with negatively pinched  sectional curvature $-\kappa^{2}\leq K\leq -1$, and Question \ref{ques:free} can be asked for this family of groups.  If in addition we know $\Gamma$ is free in Theorem \ref{thm:main}, then the constant $D(n,\kappa)$  can actually be made effective, and independent of $n$ and $\kappa$.

\begin{theorem}
\label{thm:mainfree}
Let $\Gamma<\Isom(X)$ be a finitely generated, virtually free, discrete isometry subgroup of an $n$-dimensional Hadamard manifold with pinched negative curvature $-\kappa^{2}\leq K\leq -1$. If $\delta(\Gamma)<\dfrac{1}{16}$, then $\Gamma$ is convex cocompact. 
\end{theorem}

Thus, in view of Kapovich's result \cite[Corollary 1.5]{Kap3}, we obtain,

\begin{corollary} Any finitely presented Kleinian group with $\delta(\Gamma)<\dfrac{1}{16}$ is convex cocompact. 
\end{corollary}

One of the main efforts in our proofs is to investigate the geometric properties of the quotient manifold $M=X/\Gamma$ under the condition that $\delta$ is small. While these results are only restricted to $\delta<1$, we still find that they might be of independent interest and worth highlighting. The following theorem is closely related to the classical Plateau's problem, where we obtain a certain type of linear isoperimetric inequality for the quotient manifold $M=X/\Gamma$.

\begin{theorem}
\label{thm:linear}
Suppose that $\mathcal{C}$ is a union of smooth loops in $M=X/ \Gamma$ which represents a trivial homology class in $H_1(M,\mathbb Z)$. If $\delta(\Ga)=\delta<1$, then $\mathcal C$ boundes a smooth surface $i:\Sigma\rightarrow M$ (See Definition \ref{def:surface}) whose area satisfies
\[A(i)\leq \dfrac{4}{1-\delta} \ell(\mathcal{C}),\]
where $\ell(\mathcal{C})$ denotes the total length of the smooth loops in $\mathcal{C}$.
\end{theorem}

Finitely generated Kleinian groups in dimension $3$ have only finitely many cusps \cite{Sul}, but the same result does not hold in higher dimensions \cite{Kap1}. As an application of Theorem \ref{thm:linear}, we show that under the assumption  $\delta<1$, the $\epsilon$-thin part of $M$ has only finitely many connected components when $\epsilon$ is small enough. In particular, $M$ has only finitely many cusps.

\begin{theorem}
\label{thm:bounded}
Let $\Gamma<\Isom(X)$ be a finitely generated, torsion-free, discrete isometry subgroup of an $n$-dimensional Hadamard manifold with pinched negative curvature $-\kappa^{2}\leq K\leq -1$. Suppose that $\delta(\Gamma)<1$. Then,
\begin{enumerate}
\item The number of cusps in $M=X/ \Gamma$ is no greater than the  first Betti number of $M$. 
\item $M$ has bounded geometry, that is, the noncuspidal part of $M$ has a uniform lower bound on its injectivity radius.
\item $\Gamma$ is convex cocompact if and only if the injectivity radius function $\textup{inj}: M\rightarrow \mathbb{R}$ is proper.
\end{enumerate}
\end{theorem}

\begin{remark} Note that without the assumption on the critical exponent, it is showed in \cite[Proposition 2.6]{Benoist-Hulin20} that $\Gamma$ is convex cocompact if and only if $M$ is Gromov hyperbolic and the injectivity radius function is proper.
\end{remark}

\subsection*{Outline of the proof of Theorem \ref{thm:main}} We first observe that whenever $\delta<1$ there is an area-decreasing self-map (the Besson-Courtois-Gallot map) on $M$. This allows us to prove the linear isoperimetric type inequality as in Theorem \ref{thm:linear}, from which we deduce further that closed geodesics on $M$ asymptotically have uniformly bounded normal injectivity radii. This means if there is an escaping sequence of closed geodesics on $M$, then there exists a subsequence on which the normal injectivity radii are uniformly bounded. Next, we observe that given a long closed geodesic with small normal injectivity radius, one can always separate along the normal direction to replace it by a shorter closed geodesic nearby. Then, we use the result from \cite{KL1} which states that $\Gamma$ is geometrically infinite if and only if there exists an escaping sequence of closed geodesics. The assumption that $D(n, \kappa)$ is smaller than $1/2$ excludes parabolic elements, and assume by contradiction that there is one such escaping sequence. Using the idea of infinite descent, we can reduce the length of the closed geodesics and find another escaping sequence whose lengths and normal injectivity radii are both uniformly bounded, from which we can find two loxodromic isometries that move a common point within a  uniformly bounded distance. This means the non-elementary subgroup generated by the two isometries will have large critical exponent, thus leading to a contradiction if we assume $\delta$ is small enough.

\subsection*{Organization of the paper} In Section \ref{sec:preliminary}, we review some elementary results of negatively pinched Hadamard manifolds and the Besson-Courtois-Gallot map. In Section \ref{sec:collar}, we give the proofs of Theorem \ref{thm:linear} and Theorem \ref{thm:bounded}. In Section \ref{sec:proof}, we prove Theorem \ref{thm:small-implies-proper}, which together with Theorem \ref{thm:bounded} implies Theorem \ref{thm:main} and Theorem \ref{thm:mainfree}. 
 
\subsection*{Acknowledgments} We would like to thank Grigori Avramidi, Igor Belegradek,  Lvzhou Chen, Joel Hass, Michael Kapovich, Gabriele Viaggi and  Zhichao Wang for helpful discussions. We appreciate the anonymous referees for the helpful comments and suggestions. We are also grateful to Max Planck Institute for Mathematics in Bonn for its hospitality and financial support, where this work was completed.

\section{Preliminaries}
\label{sec:preliminary}

\subsection{Discrete isometry groups}
\label{sec:isometries}
Let $X$ be a complete, simply connected, $n$-dimensional Riemannian manifold of pinched negative curvature $-\kappa^{2}\leq K\leq -1$ where $\kappa\geq 1$. The Riemannian metric on $X$ induces the distance function $d_X$ and $(X,d_X)$ is a uniquely geodesic space.  With the curvature assumption, the metric space $(X,d_X)$ is Gromov hyperbolic, where the hyperbolicity constant $\delta_{0}$ can be chosen as $\cosh^{-1}(\sqrt{2})$, i.e. every geodesic triangle in $X$ is $\delta_{0}$-slim.

By the Cartan-Hadamard theorem, $X$ is diffeomorphic to the Euclidean space $\mathbb{R}^{n}$ via the exponential map at any point in $X$. We can naturally compactify $X$ by adding the ideal boundary $\geo X$, thus the compactified space $\bar{X}=X\cup \geo X$ is homeomorphic to the unit $n$-ball $B^{n}$. 

Every isometry $\gamma\in \operatorname{Isom}(X)$ extends the action to the ideal boundary, so it induces a diffeomorphism on $\bar X$. Based on its fixed point set $\fix(\gamma)$, the isometry $\gamma$ on $X$ can be classified as follows: 

\begin{enumerate}
\item $\gamma$ is \emph{parabolic} if $\fix(\ga)$ is a singleton $\{p\}\subset \geo X$. 

\item $\gamma$ is \emph{elliptic} if it has a fixed point in $X$. In this case, the fixed point set $\fix(\ga)$ is a totally geodesic subspace of $X$ invariant under $\ga$. In particular, the identity map is elliptic. 

\item $\ga$ is \emph{loxodromic} if $\fix(\ga)$ consists of two distinct points $p, q\in \geo X$. In this case, $\ga$ stabilizes and translates along the geodesic $pq$, and we call the geodesic $pq$ the \emph{axis} of $\gamma$.

\end{enumerate}

One can also use the translation length to classify the isometries on $X$. For each isometry $\gamma\in \Isom(X)$, we define its \emph{translation length} $\tau(\gamma)$ as 
$$\tau(\gamma):=\inf_{x\in X} d_X(x, \gamma(x)). $$
The isometry $\gamma$ is loxodromic if and only if $\tau(\gamma)>0$. In this case, the infimum is attained exactly when the points are on the axis of $\gamma$. The isometry $\gamma$ is parabolic if and only if $\tau(\gamma)=0$ and the infimum is not attained. The isometry $\gamma$ is elliptic if and only if $\tau(\gamma)=0$ and the infimum is attained. 

Let $\Gamma<\Isom(X)$ be a discrete subgroup which acts on $X$ properly discontinuously. If $\Gamma
$ is torsion-free, then any nontrivial element in $\Gamma$ is either loxodromic or parabolic. We denote the quotient manifold $X/ \Gamma$ by $M$, and let $\pi: X\rightarrow M$ denote the canonical projection. The geodesic loops $c: [a, b]\rightarrow M$ at $p=c(a)=c(b)\in M$ are in  one-to-one correspondence with  geodesic segments from $x$ to $\gamma (x)$ where $x\in X$ with $\pi(x)=p$ and $\gamma\in \Gamma$. Recall that the injectivity radius at a point $p\in M$ is the largest radius for which the exponential map at $p$ is a diffeomorphism. The injectivity radius at a point $p \in M$ is half the length of shortest geodesic loop at $p$ since there are no conjugate points in $M$. We use $\textup{inj}(p)$ to denote the injectivity radius at $p$ and define
$$d_{\Gamma}(x):=\min_{\gamma\in \Gamma\setminus \{id\}} d_{X}(x, \gamma (x))$$
for $x\in X$. Then $d_{\Gamma}(x)=2\textup{inj}(\pi(x))$. We say the injectivity radius function $\textup{inj}: M\rightarrow \R$ is \emph{proper} if the preimage of a compact set is compact. The injectivity radius function is 1-Lipschitz. To see this, given any two points $p, q\in M$, let $\tilde{p}, \tilde{q}$ be a pair of lifts of $p, q$ in $X$ whose distance is the same as the distance $d(p, q)$ of $p, q\in M$. There exists an isometry $\ga\in \Ga$ such that $d_{X}(\tilde{p}, \ga{\tilde{p}})=d_{\Ga}(\tilde{p})$, and
\begin{align*}
2\textup{inj}(q)\leq d_{X}(\tilde{q}, \ga(\tilde{q})) & \leq d_{X}(\tilde{q}, \tilde{p})+d_{X}(\tilde{p}, \ga(\tilde{p}))+d_{X}(\ga(\tilde{p}), \ga(\tilde{q}))\\
& =2d(p, q)+2\textup{inj}(p).
\end{align*}
Hence, $\textup{inj}(q)-\textup{inj}(p)\leq d(p, q)$.

Recall that the \emph{critical exponent} $\delta(\Gamma)$ of a torsion-free discrete isometry group $\Ga<\Isom(X)$  is defined  to be:
$$\delta(\Gamma):=\inf \{ s: \sum_{\ga\in \Ga} \exp(-sd_X(p, \gamma(p)))<\infty \},$$
where $p$ is a given point in $X$. Note that $\delta(\Ga)$ is independent of the choice of $p$. Alternatively, one can also define the critical exponent $\delta(\Gamma)$ as follows \cite{Nicho}:
\begin{equation}
\label{for3}
\delta(\Gamma)=\limsup_{R\rightarrow \infty}\dfrac{\log(N(R))}{R},
\end{equation}
where $N(R)=\#\{\gamma\in \Gamma\mid d_{X}(x, \gamma (x))\leq R\}$ for any given point $x\in X$. 

We will need to use the following proposition later in the proofs.

\begin{proposition}\cite[Corollary 6.12]{KL1}
\label{prop:piecegeod-close-to-geod}
Let $w\in M=X/ \Gamma$ be a piecewise geodesic loop which consists of $r$ geodesic segments, and let $\alpha$ be the closed geodesic freely homotopic to $w$ such that the length $\ell(\alpha)\geq \epsilon>0$. Then $\alpha$ is contained in the $D$-neighborhood of the loop $w$, where 
\[D=\cosh^{-1}(\sqrt{2})\lceil \log_{2}r \rceil +\sinh^{-1}(2/\epsilon)+2\delta_{0}. \]
\end{proposition}

\begin{remark} The original corollary was stated under the extra assumption that $\alpha$ is simple. However, the proof of \cite[Corollary 6.12]{KL1} does not rely on this fact so we have removed the assumption here.
\end{remark}

\subsection{Thick-thin decomposition}
\label{sec:cusp}
 Given an isometry $\gamma\in \operatorname{Isom}(X)$, and a constant $\epsilon>0$, we define the \emph{Margulis region} $\operatorname{Mar}(\gamma, \epsilon)$ of $\gamma$ as
$$\operatorname{Mar}(\gamma, \epsilon): =\{ x\in X \mid d_X(x, \gamma(x))\leq \epsilon \}. $$
It is a convex subset by the convexity of the distance function. Given a point $x\in X$, and a constant $\epsilon>0$, the set 
$$\mathcal{F}_{\epsilon}(x): = \{ \gamma\in \operatorname{Isom}(X)\mid d_X(x, \gamma (x))\leq \epsilon \}$$
consists of all isometries that translate $x$ in a distance at most $\epsilon$. For any discrete subgroup $\Gamma<\operatorname{Isom}(X)$, we denote by $\Gamma_{\epsilon}(x)$ the group generated by $\mathcal{F}_{\epsilon}(x)\cap \Gamma$. The Margulis lemma \cite[Theorem 9.5]{BGS} states that $\Gamma_{\epsilon}(x)$ is a finitely generated  virtually nilpotent group  for any  $0<\epsilon< \epsilon (n, \kappa)$, where $\epsilon(n, \kappa)$ is the Margulis constant depending on the dimension $n$ of $X$ and the sectional curvature bound $\kappa$.

We define the $\Gamma$-invariant set
$$\mathcal{T}_{\epsilon}(\Gamma):=\{p \in X\mid \Gamma_{\epsilon}(p) \textup{ is infinite} \}.$$
The \emph{thin part} (more precisely, the $\epsilon$-thin part) of the quotient orbifold $M=X/ \Gamma$, which we denote by $\textup{thin}_{\epsilon}(M)$, is defined to be $\mathcal{T}_{\epsilon}(\Gamma)/ \Gamma$. The closure of the complement $M\setminus \textup{thin}_{\epsilon}(\Gamma)$ is called the \emph{thick part} of $M$, denoted by $\textup{thick}_{\epsilon}(M)$. The thin part consists of bounded and unbounded components. The bounded components are called the \emph{Margulis tubes}, which are  neighborhoods of short closed geodesics of length no greater than $\epsilon$. More precisely, for every point $x$ in the closed geodesic and every tangent vector $v$ at $x$ perpendicular to the geodesic, we consider a unit speed ray $\rho$ emanating from $x$ in the direction of $v$. There exists $R$, depending on $x$ and $v$ such that 
$$d_{\Ga}(\rho(R))=\epsilon \quad \textup{ and }  \quad d_{\Ga}(\rho(t))<\epsilon$$
for all $ t< R$. We call the arc $\rho([0, R])$ a \emph{maximal radial arc}, and a Margulis tube is the union of all radial arcs emanating from a short closed geodesic. For details, see for example \cite{BCD}.

The unbounded components are called the \emph{Margulis cusps}, which can be described more precisely as follows. Denote the fixed point set of $\Gamma$ as 
$$\fix(\Gamma):=\bigcap_{\gamma\in \Gamma} \fix(\gamma).$$
A discrete subgroup $P<\Gamma$ is called a \emph{parabolic subgroup} if $\fix(P)$ consists of a single point $\xi\in \geo X$. Given a constant $0<\epsilon< \epsilon(n, \kappa)$ and a maximal parabolic subgroup $P< \Gamma$, the set $\mathcal{T}_{\epsilon}(P)\subset X$ is precisely invariant under $P$, and we have $\stab_{\Ga}(\mathcal{T}_{\epsilon} (P))=P$, see \cite[Corollary 3.5.6]{bow3}. In this case, $\mathcal{T}_{\epsilon}(P)/ P$ can be regarded as a subset of $M$, called a Margulis cusp. The \emph{cuspidal} part of $M$ is the union of all Margulis cusps, denoted by $\textup{cusp}_{\epsilon}(M)$. Note that $\textup{cusp}_{\epsilon}(M)\subset \textup{thin}_{\epsilon}(M)$. 

In our context, the parabolic subgroups in $\Gamma$ (hence also the cuspidal part of $M$) turn out to be very simple due to the following proposition.
\begin{proposition}
\label{prop:parabolic}
Let $\Gamma<\operatorname{Isom}(X)$ be a torsion-free, discrete isometry group, and $P<\Gamma$ be any parabolic subgroup. Suppose $\delta$ is the critical exponent of $\Gamma$ and $P$ has polynomial growth rate $r$, then we have $r\leq 2\delta$. Thus, 
\begin{enumerate}
\item if $\delta<1$, then all parabolic subgroups (if they exist) are isomorphic to $\mathbb Z$.
\item if $\delta<1/2$, then all non-trivial isometries in $\Gamma$ are loxodromic.
\end{enumerate}

\begin{proof}
Let $\mathcal H$ be a horosphere that $P$ acts on and choose any basepoint $O\in \mathcal H$. Denote $d_{\mathcal H}$ the horospherical distance and $d_P$ the Cayley metric with respect to some fixed finite generating set of $P$. Then there exists a constant $C>0$ such that
\begin{equation}
\label{for1}
d_{\mathcal H}(O,\gamma (O))\leq C\cdot d_P(1,\gamma)
\end{equation}
holds for all $\gamma\in P$. By \cite[Theorem 4.6]{HIH}, there exists a constant $C'>0$ such that for any $p,q\in \mathcal H$ with $d_X(p,q)>C'$, we have
\begin{equation}
\label{for2}
d_X(p,q)\leq 2\ln\left(C'\cdot d_{\mathcal H}(p,q)\right).
\end{equation}
By possibly replacing $C$ and $C'$ by a larger constant, we might assume $C'=C$. Therefore, we obtain from the above the following asymptotic inequalities (for $R$ large)
\begin{align*}
\left|\{\gamma\in P:d_P(1,\gamma)\leq R\}\right|&\leq \left|\{\gamma\in P:d_{\mathcal H}(O,\gamma (O))\leq C\cdot R\}\right|  && \text{by } \eqref{for1} \\
&\lesssim \left|\{\gamma\in P:d_X(O,\gamma (O))\leq 2\ln(C^2\cdot R)\}\right| && \text{by } \eqref{for2} \\
&\simeq e^{2\ln(C^2\cdot R)\delta(P)}  && \text{by } \eqref{for3} \\
&\simeq R^{2\delta(P)},
\end{align*}
where $\delta(P)$ is the critical exponent of $P$. Since $\delta(P)\leq \delta$, it follows that $r\leq 2\delta$.

In particular, if $\delta<1$, then $r<2$ and by the Bass-Guivarc'h formula \cite{Bass, Guiv}, $P$ must be virtually $\mathbb Z$. But since $P$ is torsion-free, it must be $\mathbb Z$ \cite{Stal}. If $\delta<1/2$, then $r<1$ and $P$ can not exist. Thus all non-trivial elements in $\Gamma$ are loxodromic.
\end{proof}

\end{proposition}

\subsection{Geometric finiteness}
\label{sec:geometricfiniteness}

Recall that the \emph{limit set} $\Lambda(\Ga)$ of a discrete subgroup $\Gamma<\Isom(X)$ is defined to be the set of accumulation points of the $\Gamma$-orbit $\Ga(p)$ in $\geo X$, where $p$ is an arbitrary given point in $X$, and the definition is independent of the choice on $p$. If $\Lambda(\Ga)$ is finite, then $\Gamma$ is called \emph{elementary}. Otherwise, it is called \emph{nonelementary}. A point $\xi\in \Lambda(\Gamma)$ is called \emph{a conical limit point} if every geodesic ray $\rho: \mathbb{R}_{+}\rightarrow X$ asymptotic to $\xi$ projects to a non-proper map $\pi\circ\rho: \mathbb{R}_{+}\rightarrow M=X/ \Gamma$. 
We denote by $\Lambda_{c}(\Gamma)$ the set  of all conical limit points.  

We denote $\Hull(\Lambda)\subset X$ the closed convex hull of $\Lambda\subset \geo X$, which is the smallest closed convex subset in $X$ whose accumulation set in $\geo X$ is $\Lambda$, and denote $C(\Ga)=\Hull(\Lambda)/ \Ga$ the \emph{convex core} of $\Ga$. 

A discrete isometry subgroup $\Ga<\Isom X$ is \emph{geometrically finite} if the noncuspidal part of the convex core $C(\Ga)$ in $M=X/ \Ga$ is compact. Otherwise, it is called \emph{geometrically infinite}. Moreover, if $C(\Ga)$ is compact, then the discrete subgroup $\Ga$ is called \emph{convex cocompact}. 

There are various equivalent definitions of geometric finiteness, but for the interest of this paper, we will only mention one of them proved by Kapovich and the first author. For the other equivalent definitions, we refer the readers to \cite{bow3}. The following theorem is a generalization of a previous result of Bonahon \cite{Bon}.

 \begin{theorem}\cite[Theorem 1.5]{KL1}
\label{thm:KL}
A discrete subgroup $\Gamma<\Isom(X)$ is geometrically infinite if and only if there exists a sequence of closed geodesics $\alpha_{i}\subset M=X/ \Gamma$ which escapes every compact subset of $M$. 
\end{theorem}

\subsection{Admissible surfaces}\label{sec:admissible} In this section, we give a sketch on the existence of smooth admissible surfaces. This can be treated as a smooth version of \cite[Section 1.1.5]{Cal2}. In our case, we will need a slightly broader category of admissible surfaces than smooth maps, in order to include the gluing of two maps along a smooth boundary. In general the notion of piecewise smooth map is rather technical (using Whitney stratification), but in our context, we only consider  maps from a smooth surface with boundary to a smooth manifold. Thus we simplify the notion to the following:

\begin{definition}
	Given a smooth surface $\Sigma$ (possibly with boundary), and a smooth manifold $M$, we say a map $f:\Sigma\rightarrow M$ is a \emph{piecewise smooth} map if there is a smooth triangulation $\Delta=\{\sigma_1,...,\sigma_m\}$ on $\Sigma$ (i.e. edges are all smooth paths) such that,
	\begin{enumerate}
		\item $f$ is continuous,
		\item $f$ is smooth on the interior of each face $\sigma_i$,
		\item if $e= \sigma_i\cap \sigma_j$ is a common edge, then the restriction $f|_p$ is smooth.
	\end{enumerate}
\end{definition}
Roughly speaking, a piecewise smooth map is just a finite concatenation of smooth maps, possibly pleating along the gluing edges. The singular set forms a piecewise smooth $1$-skeleton on $\Sigma$. Now we return to our context that $M=X/\Ga$ is a complete pinched negatively curved manifold. Suppose $\{\eta_1,...,\eta_k\}$ is a collection of $k$ smooth loops in $M$. If there exists a set of integers $c_1,..., c_k$ such that $\sum_{i=1}^k c_i\cdot [\eta_i]=0$ in $H_1(M,\mathbb Z)$. Then we claim that $\bigcup_i c_i\eta_i$ will bound a piecewise smooth surface in the sense explained below. 

We choose a basepoint $x_0\in M$, and connect $x_0$ to each of the loop $\eta_i$ by a smooth path $p_i$. Then the loop $q_i:=p_i\ast (c_i\eta_i)\ast p_i^{-1}$ is free homotopic to $c_i\eta_i$, which also represents an element $\ga_i\in \Ga\cong \pi_1(M,x_0)$. Since $\sum_{i=1}^k c_i\cdot [\eta_i]=0$ in $H_1(M,\mathbb Z)\cong \Ga/[\Ga,\Ga]$, it follows that the product $\ga=\ga_1\cdots\ga_k$ is an element in the commutator subgroup $[\pi_1(M,x_0),\pi_1(M,x_0)]$. Thus, we can write
$$\gamma=[a_{1}, b_{1}]\cdots [a_{g}, b_{g}],$$
for some $a_i, b_i\in \Ga$. We choose  smooth loops $\alpha_i, \beta_i$ from $x_0$ such that they represent $a_i, b_i$ respectively. Fix a preimage $\widetilde x_0\in X$ of $x_0$ under the projection map $\pi:X\rightarrow M$. The loop $\sigma=\alpha_1\ast\beta_1\ast\alpha_1^{-1}\ast\beta_1^{-1}\ast\cdots\ast\alpha_g\ast\beta_g\ast\alpha_g^{-1}\ast\beta_g^{-1}\ast(q_1\ast\cdots \ast q_k)^{-1}$ is null homotopic, thus lifts to a piecewise smooth loop on $X$. Therefore, it bounds a smooth disk on $X$, that is, there exists a disk $D\subset \mathbb R^2$ and a piecewise smooth map $f:D\rightarrow X$ with $f(\partial D)=\sigma$. Moreover, by identifying $D$ with a $(4g+3k)$-polygon with the label of $\prod_{i=1}^g[\bar a_i, \bar b_i]\cdot\bar p_1\ell_1\bar p_1^{-1}\cdots\bar p_k\ell_k\bar p_k^{-1}$, we can make the map $f$ explicit by sending the edge labels $\bar a_i, \bar b_i, \bar a_i^{-1}, \bar b_i^{-1}, \bar p_i,\ell_i,\bar p_i^{-1}$ to $\alpha_i, \beta_i, \alpha_i^{-1}, \beta_i^{-1}, p_i, c_i\eta_i, p_i^{-1}$ respectively. Therefore, after gluing along the edge labels, $f$ descends to a piecewise smooth map from $\Sigma_{g,k}$ (a genus $g$ surface with $k$ boundary components) to $M$, which sends the boundary components (corresponding to $\ell_i$) to $c_i\eta_i$.

In general, we can make the following definition.

\begin{definition}\label{def:surface}
Denote a compact oriented (not necessarily connected) surface with $k$ boundary components by $\Sigma$. Given a collection of $k$ loops $\{\alpha_1,..., \alpha_k\}$ on $M$, we say a map $f: \Sigma\rightarrow M$ is \emph{admissible} with respect to $\{\alpha_1,..., \alpha_k\}$ if the following diagram commutes:
\[\begin{tikzcd}
\partial \Sigma \arrow{r}{i} \arrow[swap]{d}{\partial f} & \Sigma \arrow{d}{f} \\
\bigcup_{i=1}^k \alpha_i\arrow{r}{i} & M.
\end{tikzcd}
\]
Note that $\alpha_i$ could carry multiplicities, and the orientation of the surface $\Sigma$ induces an orientation on $\partial \Sigma$. In the above commutative diagram, we also require $\partial f$ to preserve the orientations. If there exist such $\Sigma$ and $f$, then we simply say $\bigcup_{i=1}^k \alpha_i$ bounds a surface $f$.
\end{definition}

By the above discussions, we have the following proposition.

\begin{proposition}\label{prop:admissible}
	Suppose $\{\alpha_1,...,\alpha_k\}$ is a collection of $k$ smooth loops in $M$. If there exists a set of integers $c_1,..., c_k$ such that $\sum_{i=1}^k c_i\cdot [\alpha_i]=0$ in $H_1(M,\mathbb Z)$, then there exists a piecewise smooth admissible map with respect to $\{c_1\alpha_1,...,c_k\alpha_k\}$, that is, $\bigcup_{i=1}^k c_i \alpha_i$ bounds a piecewise smooth surface $f:\Sigma\rightarrow M$.
\end{proposition}




Given two Riemannian manifolds $N,M$, a smooth map $F:N\rightarrow M$ and a positive integer $\mathfrak p\leq \min\{\dim(N),\dim{M}\}$, the $\mathfrak p$-Jacobian of $F$ at a point $x\in N$ is defined to be
\[\operatorname{Jac}_{\mathfrak{p} }(F)(x)=\sup||dF_x(e_1)\wedge dF_x(e_2)\wedge...\wedge dF_x(e_\mathfrak p)||,\]
where the supremum is taken over all orthonormal $\mathfrak p$-frames $\{e_1,...,e_{\mathfrak p}\}$ on $T_xN$, and the norm is induced by the Riemannian inner product at $T_{F(x)}M$. Note that in the case $\mathfrak p=\dim N\leq \dim M$, the $\mathfrak p$-Jacobian of $F$ coincides with $\sqrt{\det_{g_N}F^*g_M}$.
\begin{definition}
	Given a Riemannian manifold $M$, a smooth map $f:\Sigma\rightarrow M$ and a smooth region $U\subset \Sigma$, we define the area of the map on $U$ to be
	\[A(f|_U):=\int_{U}|\operatorname{Jac}_2 f|(x)dV_{\Sigma}.\]
where $dV_{\Sigma}$ is the volume form on $\Sigma$ with respect to some chosen Riemannian metric $g_\Sigma$, and it is clear the definition of area is independent on the choice of $g_\Sigma$. When $U=\Sigma$, we simply denote it by $A(f)$. The definition naturally extends to a piecewise smooth map. Note that, at the region where $df$ is degenerate, $(\operatorname{Jac}_2 f)$ vanishes, so it does not contribute to the area.
\end{definition}

\subsection{Besson-Courtois-Gallot map}\label{sec:BCG}
In this section, we give a brief introduction to the Besson-Courtois-Gallot map and we refer the readers to \cite{BCG} for a more detailed exposition. First we recall that given any discrete subgroup $\Gamma<\operatorname{Isom}(X)$, there exists a family of positive finite Borel measures called the Patterson-Sullivan measures, which satisfy:
\begin{enumerate}
\item $\mu_x$ is $\Gamma$-equivariant, for all $x\in X$,
\item $d\mu_x(\theta)=e^{-\delta B(x,\theta)}d\mu_o(\theta)$, for all $x\in X$, and
$\theta\in \partial_{\infty}X$,
\end{enumerate}
where $\delta$ is the critical exponent of $\Gamma$, $o$ is a basepoint on $X$, and $B(x,\theta)$ is the Busemann function on
$X$ with respect to $o$. Recall that, the Busemann function $B$ is defined by
$$B(x,\theta)=\lim_{t\rightarrow \infty}\left(d(x,\alpha_\theta(t))-t\right)$$
where $\alpha_\theta(t)$ is the unique geodesic ray from $o$ to $\theta$.

We note that the Busemann function $B(x,\theta)$ is convex on $X$. If $\mu$ is any finite Borel measure supported on at least two points on
$\partial_\infty X$, then the following function
\[x\mapsto \mathcal{B}_{\mu}(x):=\int_{\geo X}e^{B(x,\theta)}d\mu(\theta)\]
is strictly convex, and one can check it tends to $+\infty$ as $x\rightarrow \partial_\infty X$.
Hence we can define the barycenter $\operatorname{bar}(\mu)$ of $\mu$ to be the unique point in $X$ where the function attains its minimum.

Now we construct the following map $\tilde F:X\rightarrow X$ that is given by
\[x\mapsto \operatorname{bar}(e^{-B(x,\theta)}\mu_x),\]
and $e^{-B(x,\theta)}\mu_x$ denotes the unique (up to measure zero) Borel measure which is absolutely continuous with respect to $\mu_x$, with the corresponding Radon-Nikodym derivative $e^{-B(x,\theta)}$.

\begin{theorem}[Besson-Courtois-Gallot\cite{BCG}]\label{thm:BCG}
The map $\tilde F: X\rightarrow X$ constructed above satisfies the following conditions:
\begin{enumerate}
\item $\tilde F$ is $\Gamma$-equivariant, and thus descends to a map $F:M\rightarrow M$.
\item $F$ is smooth and homotopic to the identity.
\item $|\operatorname{Jac}_{\mathfrak{p} }(F)(x)|\leq \left(\frac{1+\delta}{\mathfrak p}\right)^{\mathfrak p}$ for any integer $\mathfrak p\in [1,\dim M]$ and any $x\in M$.
\end{enumerate}
\end{theorem}

\begin{remark} The case of $\mathfrak p=1$ in $(3)$ is not directly stated in the paper, however it is clear from the 2-form equation \cite[Equation 4.11]{BCG} that $||dF||\leq (1+\delta)$. According to the theorem, if $\delta\leq \mathfrak p-1$, then $|\operatorname{Jac}_{\mathfrak p}(F)|\leq 1$ hence $F$ is a $\mathfrak p$-dimensional volume-decreasing map. However, in order to obtain the linear isoperimetric inequality in Section \ref{sub:linear-isom}, we will need an area-decreasing map, which is assured only in the case $\delta<1$. Thus, we will only apply the theorem to the cases $\mathfrak p=1,2$.
\end{remark}

\subsection*{Notations} In the rest of the paper, $X$ always denotes a negatively pinched Hadamard manifold with sectional curvature $-\kappa^{2}\leq K\leq -1$, and $\Gamma<\Isom(X)$ denotes a torsion-free discrete isometry subgroup. Let $M=X/ \Gamma$ be the quotient manifold, $\pi: X\rightarrow M$ be the quotient map, and $d$ be the distance on $M$. Let $\delta$ denote the critical exponent of $\Gamma$ and $C(\delta)=4/(1-\delta)$. We use $\ell$ and $A$ to denote the length and area function respectively. We let $\textup{inj}(x)$ denote the injectivity radius at a point $x\in M$, and let $\textup{NJ}(S)$ denote the normal injectivity radius of a submanifold $S\subset M$ (see Section \ref{sub:collar}). 
 
\section{Geometry with small critical exponent}
\label{sec:collar}

In this section, we investigate the geometry of the quotient manifold $M$ under the assumption $\delta<1$.

\subsection{Linear isoperimetric type inequality}\label{sub:linear-isom} The study of isoperimetric problem has a great long history. In the classical context, given a region $\Omega\subset \mathbb R^2$, it is natural to ask what is the optimal relation between its area $A(\Omega)$ and the length of its bounding curve $\ell(\partial \Omega)$. It is proved that there is a quadratic relation $A(\Omega)\leq \ell(\partial \Omega)^2/4\pi$, and the equality holds if and only if $\Omega$ has a circular boundary. However, the main interest of this paper has driven us to consider in a slightly different context. Let $M=X/\Gamma$ be a complete quotient manifold and $\mathcal{C}\subset M$ be a union of smooth loops which represents a trivial homology class in $M$. By the discussion in Section \ref{sec:admissible}, $\mathcal{C}$ bounds an admissible surface. Among all admissible surfaces, we find one surface $\Sigma$ such that $A(\Sigma)$ and $\ell(\partial \Sigma)$ satisfy a linear isoperimetric type inequality. 

\begin{definition}
We say a family of loops $\mathcal F=\{\alpha_1,...,\alpha_k\}$ in $M$ is \emph{irreducible} if either
\begin{enumerate}
\item $k=1$ and $\alpha_1$ represents a trivial or torsion homology class, or
\item $\mathcal F$ consists of linearly dependent loops, and any non-trivial subfamily of $\mathcal F$ is linearly independent.
\end{enumerate}
\end{definition}

Suppose $\mathcal F=\{\alpha_1,...,\alpha_k\}$ is an irreducible family of loops. In case $(1)$, $\mathcal F$ consists of one homology class $[\alpha]$, so there is a minimal positive integer $c$ such that $c\cdot[\alpha]=0$. In case $(2)$, there exists a unique (up to a sign) set of integers $c_1,...,c_k$ such that $\gcd(c_1,...,c_k)=1$ and $\sum_{i=1}^k c_i\cdot[\alpha_i]=0$ in $H_1(M)$. Thus, there exist admissible surfaces in $M$ with respect to  $c\cdot[\alpha]$ (or $\bigcup_{i=1}^k c_i \alpha_i$) and by irreducibility they are necessarily connected. Note that $c_i\alpha_i$ denotes the $c_i$ multiple of $\alpha_i$ and when $c_i$ is negative, it means to reverse the orientation of $\alpha_i$. We call the set of integers $c_1,...,c_k$ (or $c$ if in case $(1)$) the \emph{associated integers} of the irreducible family.

\begin{figure}
	\centering
	\includegraphics[width=3.5in]{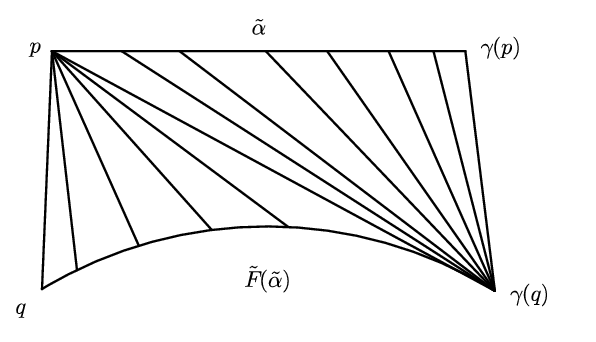}
	\caption{ \label{F2}}
\end{figure}

\begin{theorem}\label{thm:linear iso}
Let $\mathcal F=\{\alpha_1,...,\alpha_k\}$ be any family of smooth loops in $M$ which are linearly dependent in $H_1(M,\mathbb Z)$
such that there are integers $c_1,...,c_k$ satisfying $\sum_{i=1}^k c_i\cdot[\alpha_i]=0$ in $H_1(M)$. Suppose the critical exponent $\delta<1$. Then $\bigcup_{i=1}^k c_i \alpha_i$ bounds a smooth surface $f_0:\Sigma\rightarrow  M$ whose area satisfies
\[A(f_0)\leq \frac{4}{1-\delta}\ell(f_0(\partial \Sigma))= \frac{4}{1-\delta}\left(\sum_{i=1}^k |c_i| \ell(\alpha_i)\right).\]
\end{theorem}

\begin{proof}
It is sufficient to assume $\mathcal F$ is irreducible. Otherwise, we decompose $\mathcal F$ into irreducible subfamilies and use the additivity of area and length functions on disjoint unions. We consider the set $\mathfrak S$ which consists of all piecewise smooth surfaces bounded by $\bigcup_{i=1}^k c_i \alpha_i$, or more precisely, we set
\[\mathfrak S:=\{f:\Sigma\rightarrow M \mid f \text{ is a piecewise smooth admissible map with respect to } \{c_1\alpha_1,\cdots, c_k\alpha_k\}\}.\]
By Proposition \ref{prop:admissible} it is non-empty. Let $A_0=\inf\{A(f):f\in \mathfrak S\}$. To avoid possible existence and regularity issues (see the following remark) of minimal surfaces in $M$, we can choose a piecewise smooth admissible map $f_\epsilon\in \mathfrak S$ such that $A(f_\epsilon)\leq (1+\epsilon)A_0$ for any $\epsilon>0$. Composing with the Besson-Courtois-Gallot map $F$ as described in Section \ref{sec:BCG}, we obtain a piecewise smooth admissible map $F\circ f_\epsilon$ with respect to $\bigcup_{i=1}^k c_i F(\alpha_i)$. By Theorem \ref{thm:BCG} we have the area estimate
\begin{align*}
\label{for:area}
A\left(F\circ f_\epsilon\right)=\int_{\Sigma}|\operatorname{Jac}_2(F\circ f_\epsilon)|dV_\Sigma&\leq \int_{\Sigma}|\operatorname{Jac}_2 F|\cdot |\operatorname{Jac}_2 f_\epsilon|dV_\Sigma\\
&\leq \left(\frac{1+\delta}{2}\right)^2 A(f_\epsilon)\\
&\leq \left(\frac{1+\delta}{2}\right)^2 (1+\epsilon)A_{0},
\end{align*}
and the length estimate $\ell\left(F(\alpha_i)\right)\leq (1+\delta)\ell(\alpha_i)$. For each $\alpha_i$, since $F(\alpha_i)$ is free homotopic to $\alpha_i$, we can build an (immersed) cylindrical homotopy $\Sigma_i\subset M$ between them by taking the image of union of two geodesic cones $\operatorname{Cone}_p\left(\tilde F(\tilde \alpha)\right)$ and $\operatorname{Cone}_{\gamma(q)}\left(\tilde \alpha\right)$ under the projection $\pi: X\rightarrow M$, see Figure \ref{F2}. Here $\gamma\in\Gamma$ is an element represented by $\alpha$, $\tilde{\alpha}$ is a lift of $\alpha$, and  both $p,q$ and $\gamma(p),\gamma(q)$ are connected by geodesics. To estimate the area of $\Sigma_i$, we will need the following lemma.

\begin{lemma} For any $p\in X$ and any smooth curve $\alpha\subset X$, the geodesic cone $\operatorname{Cone}_p(\alpha)$ has an area bound
\[A\left(\operatorname{Cone}_p(\alpha)\right)\leq \ell(\alpha).\]
\end{lemma}

\begin{proof}
We parameterize the smooth curve by $\alpha:[0,1]\rightarrow X$, and denote $D(s)=d(p,\alpha(s))$. The geodesic cone
$\operatorname{Cone}_p(\alpha)$ can be parameterized by the smooth map 
\[\Phi:[0,1]\times [0,D(s)]\rightarrow X\]
\[(s,t)\mapsto \operatorname{exp}_p\left(t\beta(s)\right)\]
where $\beta(s)$ is the unit vector in the direction  of the preimage of $\alpha$ under the exponential map, i.e. the unique curve in $T_pX$ satisfying $\operatorname{exp}_p\left(D(s)\beta(s)\right)=\alpha(s)$. Since $\alpha(s)=\Phi(s,D(s))$ we have 
\[\alpha'(s)=\left[\frac{\partial\Phi}{\partial s}+\frac{\partial\Phi}{\partial t}\cdot D'(s)\right]\left(s,D(s)\right).\]
Let $\gamma_s(t)=\Phi(s,t)$.  For each $s$, it is a unit speed geodesic connecting $p$ to $\alpha(s)$, so at any point $(s,t)\in [0,1]\times [0,D(s)]$, we have that
\[\frac{\partial\Phi}{\partial t}=\gamma_s'(t),\quad \frac{\partial\Phi}{\partial s}=J_s(t),\]
where $J_s(t)$ is the unique Jacobi field along $\gamma_s$ satisfying $J_s(0)=0$ and $J_s(D(s))=\frac{\partial\Phi}{\partial s}(s,D(s))=\alpha'(s)-\gamma_s'(D(s))\cdot D'(s)$, which is  the projection of $\alpha'(s)$ orthogonal  to $\gamma_s'(D(s))$. This implies that $J_s(t)$ is a normal Jacobi field and that $\frac{\partial\Phi}{\partial t}\perp \frac{\partial\Phi}{\partial s}$. Therefore, we obtain \[|\operatorname{Jac}(\Phi)|=||\frac{\partial\Phi}{\partial s}\wedge\frac{\partial\Phi}{\partial t}||=||\frac{\partial\Phi}{\partial s}||\cdot||\frac{\partial\Phi}{\partial t}||=||J_s(t)||.\]
Using \cite[Proposition 2.3]{HIH} and the curvature assumption $K\leq -1$, we can estimate the norm of the Jacobi fields by
\begin{equation}
\label{for4}
||J_s(t)||\leq \frac{\sinh t}{\sinh (D(s))}\cdot ||J_s(D(s))||\leq \frac{\sinh t}{\sinh (D(s))}\cdot ||\alpha'(s)||.
\end{equation}
Finally we obtain the area estimate of the geodesic cone by
\begin{align}
\begin{aligned}	
A\left(\operatorname{Cone}_p(\alpha)\right)&\leq \int_0^1\int_0^{D(s)}|\operatorname{Jac}(\Phi)|\;dt\;ds\\
&\leq \int_0^1\int_0^{D(s)} \frac{\sinh t}{\sinh (D(s))}\cdot ||\alpha'(s)||\;dt\;ds && \text{by } \eqref{for4} \\
&\leq \int_0^1 ||\alpha'(s)||\;ds \\
&\leq \ell(\alpha).
\end{aligned}
\label{for:area}
\end{align}
\end{proof}

Now we continue with the proof. By the lemma above, we have 
\begin{equation}
\label{for:areacone}
A(\Sigma_i)\leq \ell(\alpha_i)+\ell\left(F(\alpha_i)\right)\leq (2+\delta)\ell(\alpha_i).
\end{equation}
Here $\Sigma_i$ is a piecewise immersed surface in $M$ and we can choose any piecewise smooth parametrization $\sigma_i: S^1\times [0,1]\rightarrow M$ to represent $\Sigma_i$. If we concatenate each $\sigma_i$ with $F\circ f_\epsilon$ (glue $\bigcup_{i=1}^k c_i\Sigma_i$ onto $F\circ f_\epsilon(\Sigma)$ on $M$), we get a new piecewise smooth admissible surface $f_\epsilon'$ with respect to $\bigcup_{i=1}^k c_i \alpha_i$, and by the assumption $A(f_\epsilon')\geq A_0$. On the other hand, combining the above inequalities we have
\begin{align*}
A_{0}&\leq A(f_\epsilon')=A\left(F\circ f_\epsilon\right)+\sum_{i=1}^k |c_i|\cdot A(\Sigma_i)\\
&\leq \left(\frac{1+\delta}{2}\right)^2(1+\epsilon)A_0+(2+\delta)\left(\sum_{i=1}^k |c_i|\ell(\alpha_i)\right)  && \text{by } \eqref{for:area} \text{ and } \eqref{for:areacone}.
\end{align*}
Thus by letting $\epsilon$ tend to zero, we obtain
\[A_0\leq \frac{4(2+\delta)}{(1-\delta)(3+\delta)}\left(\sum_{i=1}^k |c_i| \ell(\alpha_i)\right)<\frac{4}{1-\delta}\left(\sum_{i=1}^k |c_i| \ell(\alpha_i)\right).\]
Therefore, we can always choose a piecewise smooth map in $\mathfrak S$ whose area is arbitrarily close to $A_0$ and finally we can always smoothen it with arbitrarily small increase on the area. In particular, there is a smooth admissible map $f_0$ with area
\[A(f_0)\leq \frac{4}{1-\delta}\left(\sum_{i=1}^k |c_i| \ell(\alpha_i)\right).\]
\end{proof}

\begin{remark} The existence and regularity of minimal surfaces for a general complete manifold relate to the generalized Plateau's problem, which has been studied in \cite{Morrey48}. If there is a uniform lower bound on the injectivity radius on $M$, then the condition of ``homogeneously regular'' in \cite{Morrey48} is satisfied hence the existence and regularity of the area minimizer hold. Although later in Theorem \ref{thm:collar} we  manage to show $M$ has bounded geometry, yet the proof relies on this theorem, hence using this will fall into a circular reasoning.

We do not pursue the optimal bound in the theorem above. Indeed, the linear isoperimetric constant we produce via this method will always tend to infinity as $\delta\rightarrow 1$. This stands as an obstacle in improving our main theorems as $\delta$ approaches $1$.
\end{remark}

\subsection{Asymptotically uniformly bounded tubular neighborhood}
\label{sub:collar}
Let $S$ be a closed submanifold of $M$. Denote $N(S,M)=\{(x,v)\in TM:x\in S\;\textit{and}\;v\perp T_xS\}$ the \emph{normal bundle} of $S$ in $M$, and denote $N_r(S,M)=\{(x,v)\in N(S,M):|v|<r\}$ the \emph{$r$-normal bundle} of $S$ in $M$. The \emph{normal exponential map} $\operatorname{exp}_{S}$ is defined to be the restriction of the exponential map $\operatorname{exp}: TM\rightarrow M$ to the normal bundle $N(S, M)$ of $S$ in $M$. The \emph{normal injectivity radius} $\operatorname{NJ}(S)$ is defined to be the supremum of $r$ such that $\operatorname{exp}_{S}$ is an embedding on $N_{r}(S, M)$.  In the case where $r\leq\operatorname{NJ}(S)$, we say $\operatorname{exp}_{S}(N_{r}(S, M))=\{x\in M\mid d(x, S)< r\}$ is the \emph{$r$-tubular neighborhood} of $S$ in $M$, and we denote it by $T_r(S)$. By convention, if the submanifold has a self-intersection, we declare that it has normal injectivity radius zero.

\begin{lemma}\label{lem:area-lower-bound}
	Let $\alpha$ be a closed geodesic in $M$ with $\operatorname{NJ}(\alpha)=R>0$, and let $T_R(\alpha)$ be its $R$-tubular neighborhood in $M$. If $i:\Sigma\rightarrow M$ is any smooth admissible map with respect to $\{k\alpha, \alpha'\}$ such that either $\alpha'$ is empty, or $\alpha'$ consists of a union of smooth loops outside of $T_R(\alpha)$ (i.e. $d_M(\alpha',\alpha)> R$). Then
	\[A(i|_{i^{-1}(T_R(\alpha))})\geq kR\cdot\ell(\alpha).\]
\end{lemma}

\begin{proof}
	We choose a Riemannian metric $g_0$ on $\Sigma$, and let $\epsilon_1, \epsilon_2$ be two positive real numbers recognized to be small and to be determined later. First, we perturb the pullback metric $i^*g_M$ to be Riemannian on $\Sigma$ by setting $g=i^*g_M+\epsilon_1 g_0$ and use this to estimate the area of $i$. It follows that for any $\epsilon>0$, and any region $U\subset \Sigma$, we have
	\begin{align}
		\begin{aligned}
		|\operatorname{vol}_g(U)-A(i|_U)|&=\left| \int_U 1 dV_{g}-\int_{U} |\operatorname{Jac}_2 i|dV_{g_0}\right|\\
		&=\int_U\left(\sqrt{\operatorname{det}_{g_0}(g)}- \sqrt{\operatorname{det}_{g_0}(i^*{g_M})}\right) dV_{g_0}\\
		&\leq \int_\Sigma\left(\sqrt{\operatorname{det}_{g_0}(g)}- \sqrt{\operatorname{det}_{g_0}(i^*{g_M})}\right) dV_{g_0}\\
		&< \epsilon,
		\end{aligned}\label{ineq:areacomp}
	\end{align}
after choosing $\epsilon_1$ small enough. Note that this follows from the continuity of determinant function, and the estimate is uniform on $U$.

Next, we choose a suitable function on $\Sigma$ and use the coarea formula to estimate $\operatorname{vol}_g(U)$. Denote $\sigma\subset\partial \Sigma$ the boundary component which sends to $k\alpha$ under $i$, and denote $\rho_\alpha:M\rightarrow \mathbb R$ the distance function to $\alpha$ on $M$. Now we construct a function $f:\Sigma\rightarrow \mathbb R$ by setting
\[f=\rho_\alpha\circ i+\epsilon_2 \varphi,\]
where $\varphi$ is a smooth function on $\Sigma$ chosen so that
\begin{enumerate}
	\item $\varphi(x)=0$ on $\sigma$ and $\varphi(x)>0$ on $\Sigma\backslash \sigma$.
	\item there exists a collar neighborhood $V$ of $\sigma$ such that $d\varphi(x)\neq 0$ when $x\in V\backslash \sigma$.
\end{enumerate}
For example, one can choose $\varphi$ to be the distance function to $\sigma$ on its local neighborhood and then extend smoothly to any positive function outside. For this choice, it is clear that $f(x)\geq 0$ and $f^{-1}(0)=\sigma$. Since $M$ is negatively curved, there is no conjugate point for $M$. Thus for any $y\in T_R(\alpha)$, there is a unique geodesic projection onto $\alpha$, so $\rho_\alpha$ is smooth on $T_R(\alpha)\backslash \alpha$. It follows that $f$ is smooth on $i^{-1}(T_R(\alpha))\backslash \sigma\subset \Sigma$. We can estimate the norm of its differential with respect to metric $g$.
\begin{align}
	\begin{aligned}
	||df||&=||d\rho_\alpha\circ di+\epsilon_2 d\varphi||\\
	&\leq ||d\rho_\alpha||\cdot ||di||+\epsilon_2 ||d\varphi||&& \text{note that }i \text{ is $1$-Lipschitz}\\ 
	&< (1+\epsilon),
	\end{aligned}
\label{ineq:df}
\end{align}
after choosing $\epsilon_2$ small enough, and this uses the compactness of $\Sigma$. 

Finally we estimate the area of $i$ on $i^{-1}(T_R(\alpha))$. By the construction of $f$, we have $f^{-1}\left([0,R)\right)\subset i^{-1}(T_R(\alpha))$. Thus if we set $U=f^{-1}\left([0,R)\right)$, then $\operatorname{vol}_g(U)\leq \operatorname{vol}_g(i^{-1}(T_R(\alpha)))$. On the other hand, by the coarea formula \cite[Section 13.4]{BuragoZalgaller}, we obtain from \eqref{ineq:df} that
\begin{align}\label{ineq:coarea}
	\begin{aligned}
		\operatorname{vol}_g(U)&>\frac{1}{1+\epsilon} \int_{U}||df||dV_{g}\\
		&=\frac{1}{1+\epsilon} \int_{0}^R\ell_g(f^{-1}(t))dt && \text{coarea formula}.
	\end{aligned}
\end{align}
Note that in the above formula, $f^{-1}(t)$ might not be a smooth curve if $t$ is a singular value. But by Sard's theorem, almost all values $r\in (0,R)$ are regular, in which case the level sets are unions of smooth circles on $\Sigma$, and $\ell_g$ denotes the total length of the circles. In particular, the above integral makes sense. Other boundary components (if any) of $\Sigma$ do not intersect with $i^{-1}(T_R(\alpha))$ by the assumption, so given any regular value $t\in [0,R)$,  $f^{-1}(t)$ (up to orientation) is homologous to $f^{-1}(0)=\sigma$ on $\Sigma$. Hence taking their images in $M$, we obtain that $i(f^{-1}(t))$ (also a union of smooth loops) is homologous to $k\alpha$ on $M$. Since  they are entirely contained in $T_R(\alpha)$, $i(f^{-1}(t))$ is in fact free homotopic to $k \alpha$. More precisely, for almost all $t\in (0,R)$, if we write $i(f^{-1}(t))$ as a disjoint union of circles $\bigcup_{i=1}^m \alpha_{i}$, then each $\alpha_i$ is a smooth loop free homotopic to $k_i\alpha$ for $k_i\in \mathbb Z$, since the fundamental group of the $R$-neighborhood of $\alpha$ is a cyclic group generated by the loop $\alpha$. (Some $k_{i}$ could be zero in which case $\alpha_{i}$ is homotopically trivial in $M$.) Moreover, we have $\sum_{i=1}^m k_i=k$. Since $\alpha$ is a closed geodesic, we have that $\ell\left(i(f^{-1}(t))\right)=\sum_{i=1}^m \ell(\alpha_i)\geq \sum_{i=1}^m |k_i|\ell(\alpha) \geq k\ell(\alpha)$. Note that $i$ is $1$-Lipschitz, so we have $\ell_g(f^{-1}(t))\geq \ell\left(i(f^{-1}(t))\right)$. Combining the above inequality with \eqref{ineq:areacomp} and \eqref{ineq:coarea}, we have
\[A(i|_{i^{-1}(T_R(\alpha))})>\frac{1}{1+\epsilon}kR\cdot \ell(\alpha)-\epsilon.\]
Since $\epsilon>0$ is arbitrary, the lemma follows.
\end{proof}

\begin{lemma}
\label{lem:lowerbound}
Given $N$ cusps in $M$ and a constant $\epsilon>0$ small such that $\{M_{12\epsilon}^{(i)}:1\leq i\leq N\}$ are disjoint components of the cuspidal part $\textup{cusp}_{12\epsilon}(M)$. Suppose $\iota: \Sigma\rightarrow M$ bounds an irreducible collection of smooth loops $\bigcup_{i=1}^N c_i\alpha_i$, where each $\alpha_i$ is contained in the $2\epsilon$-thinner part $M^{(i)}_{2\epsilon}\subset M_{12\epsilon}^{(i)}$ in each cusp component and is homologically non-trivial. Then 
$$A(\iota)\geq 4\epsilon^{2}.$$

\end{lemma}

\begin{proof} Since the collection is irreducible and $\alpha_1$ is homologically non-trivial in its cusp component (which might be homologically trivial in $M$), $\iota(\Sigma)$ has to leave $M_{12\epsilon}^{(1)}$. We will only focus on the region $U_0:=\iota^{-1}(M^{(1)}_{12\epsilon})$ as shown in Figure \ref{F7}. If we denote $M_{4\epsilon}^{(1)}\subset M_{12\epsilon}^{(1)}$ the $4\epsilon$-thinner part and $T_1=M_{12\epsilon}^{(1)}\backslash M_{4\epsilon}^{(1)}$, then certainly we have
	\[A(\iota)\geq A(\iota|_{i^{-1}(T_1)}).\]
So it suffices to give a lower bound on the area restricted to $T_1$ region.
	
Similar to the proof of Lemma \ref{lem:area-lower-bound}, we first choose the same perturbed Riemannian metric on $\Sigma$ as $g=\iota^* g_M+\epsilon_1 g_0$, and for any $\epsilon'>0$, the estimate of \eqref{ineq:areacomp} still works after choosing $\epsilon_1$ small enough. Thus, we have for any $U\subset \Sigma$, 
	\begin{equation}\label{ineq:volcomp}
		|\operatorname{vol}_g(U)-A(\iota|_U)|<\epsilon'.
	\end{equation}
	Denote $\sigma\subset \partial \Sigma$ the boundary component which maps to $c_1\alpha_1$ under $\iota$, and let $\varphi$ be as before the smooth function on $\Sigma$ such that
	\begin{enumerate}
		\item $\varphi(x)=0$ on $\sigma$ and $\varphi(x)>0$ on $\Sigma\backslash \sigma$.
		\item there exists a collar neighborhood $V$ of $\sigma$ such that $d\varphi(x)\neq 0$ when $x\in V\backslash \sigma$.
	\end{enumerate}
	We choose a smooth approximation (\cite[Proposition 2.1]{GreeneWu}) of the injectivity radius function on a neighborhood of $\iota(\Sigma)$, denoted by $j$, such that 
	\begin{enumerate}
		\item $j>0$ on $\iota(\Sigma)$,
		\item $j$ is $(1+\epsilon')$-Lipschitz, and
		\item $|j(y)-\operatorname{inj}(y)|<\epsilon$ on $\iota(\Sigma)$.
	\end{enumerate}
	Choose a smooth bump function $0\leq \psi\leq 1$ on $\Sigma$ such that $\psi=1$ on $\iota^{-1}(T_1)$ and $\psi=0$ on $\sigma$. Since $\Sigma$ is compact, there exists $\mathcal K>0$ such that $||\varphi||<\mathcal K$ and $||d\varphi||<\mathcal K$. We choose a positive constant  $\epsilon_2<\min\{\epsilon, \epsilon'\} / \mathcal{K}$. Now we define the smooth function $f:\Sigma\rightarrow \mathbb R$ by
	\[f=\epsilon_2\varphi+\psi\cdot (j\circ \iota).\]	
By the construction of $f$, we see that $f(x)\geq 0$ on $U_0$ and $f^{-1}(0)=\sigma$. When restricting to $U_1:=\iota^{-1}(T_1)=\iota^{-1}(M_{12\epsilon}^{(1)}\backslash M_{4\epsilon}^{(1)})$, the norm of its differential under metric $g$ can be estimated by
	\begin{align}\label{ineq:dfU}
		\begin{aligned}
			||df||_{U_{1}}&=||\epsilon_2 d\varphi+dj\circ d\iota||\\
			&\leq \epsilon_2||d\varphi||+||dj||\cdot||d\iota||\\
			&<(1+2\epsilon').
		\end{aligned}
	\end{align}
 The first inequality follows from the fact that $\psi=1$ on $\iota^{-1}(T_1)$, and the last inequality uses 1-Lipschitzianity of $\iota$ and also the choice of $j$ and $\epsilon_2$. 
	 Now we investigate the value of $f$ on $U_0$, and apply the coarea formula to give a lower bound for the area of $\iota|_{f^{-1}([4\epsilon,5\epsilon])\cap U_0}$.
	 
\begin{figure}
\centering
\includegraphics[width=3.0in]{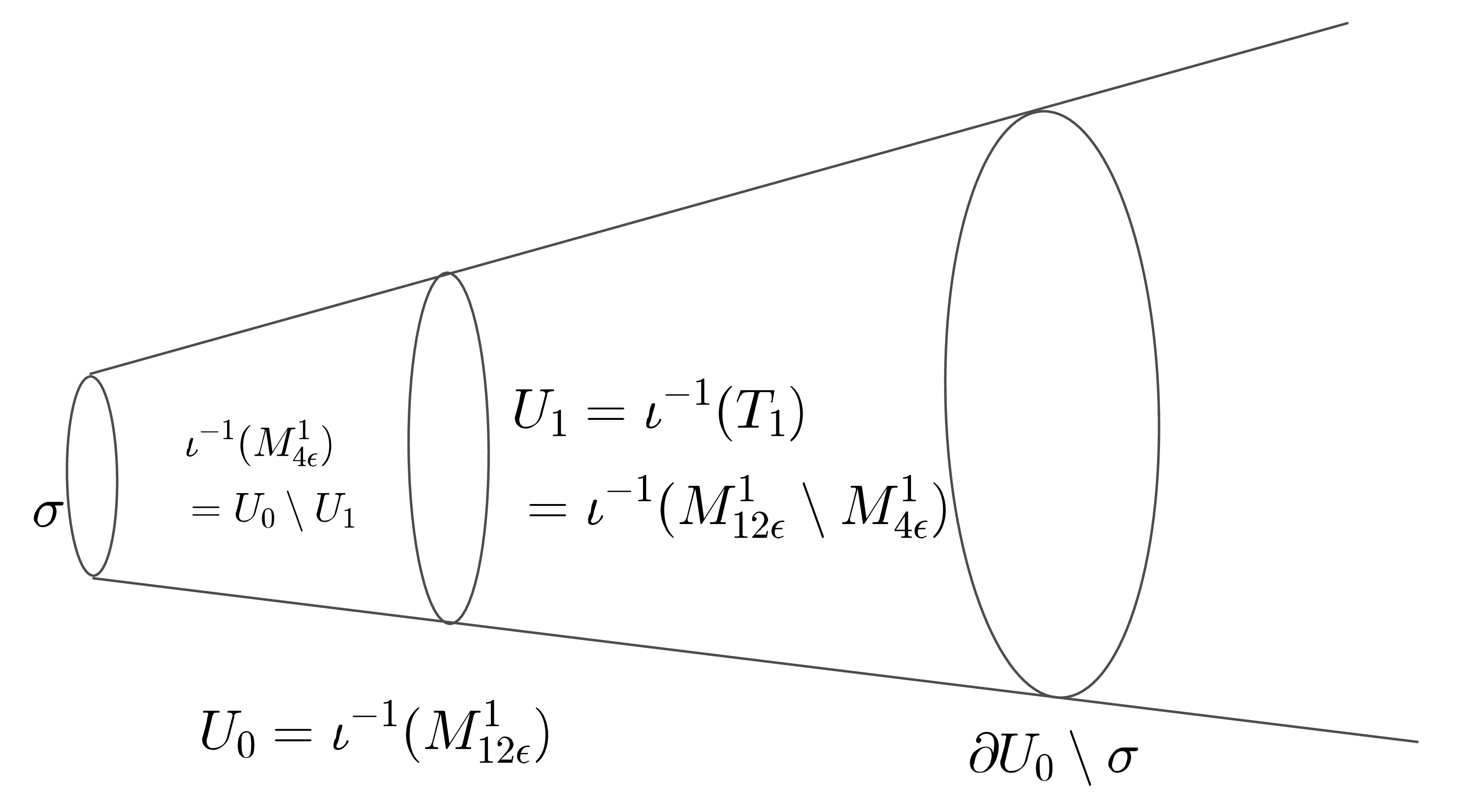}
\caption{ \label{F7}}
\end{figure}

\begin{claim}
The subset $f^{-1}([4\epsilon,5\epsilon])\cap U_0$ is contained in  $U_1$, and $f^{-1}([0,5\epsilon])\cap U_0$ is disjoint from $\partial U_0\backslash \sigma$. 
\end{claim}
\begin{proof}
  For any $x\in U_0\backslash U_1=\iota^{-1}(M^{(1)}_{4\epsilon})$, we have
	\begin{align*}
		f(x) &=\epsilon_2 \varphi(x)+\psi(x)\cdot j(\iota(x))\\
		&<\epsilon+j(\iota(x))\\
		&<\epsilon+\operatorname{inj}(\iota(x))+\epsilon\\
		&<4\epsilon.
	\end{align*}
This implies that $f^{-1}([4\epsilon,5\epsilon])\cap U_0$ is contained in  $U_1$. Second, we notice that $\partial U_0$ consists of $\sigma$ and other boundary components on which $\operatorname{inj}=6\epsilon$. For any $x\in \partial U_0\backslash \sigma$, we have
	\begin{align*}
		f(x)&=\epsilon_2 \varphi(x)+\psi(x)\cdot j(\iota(x))\\
		&>j(\iota(x))\\
		&>\operatorname{inj}(\iota(x))-\epsilon\\
		&>5\epsilon.
	\end{align*}
	This shows that for any $t\in[0,5\epsilon]$, $f^{-1}(t)$ (restricted on $U_0$) does not intersect with $\partial U_0$.
\end{proof}

As a consequence, for any regular values $t\in [0,5\epsilon]$, $f^{-1}(t)$ is a union of smooth loops that cobounds with $f^{-1}(0)=\sigma$, and in particular is homologous to $\sigma$. Under the image of $\iota$, it shows that $\iota(f^{-1}(t)\cap U_0)$ is homologous to $\iota(\sigma)=c_1[\alpha_1]\neq 0$. Moreover, for regular values $t\in (4\epsilon, 5\epsilon)$, and any point $y\in \iota(f^{-1}(t)\cap U_0)$, we let $x\in f^{-1}(t)\cap U_0\subset U_1$ be any preimage of $y$.  Then we have
	\begin{align*}
		\operatorname{inj}(y)&=\operatorname{inj}(\iota(x))\\
		&\geq j(\iota(x))-\epsilon\\
		&=f(x)-\epsilon_2\varphi(x)-\epsilon && \psi(x)=1 \text{ since } x\in U_1\\
		&\geq t-2\epsilon\\
		&> 2\epsilon.
	\end{align*}
	In particular, $\ell(\iota(f^{-1}(t)\cap U_0))\geq 2\operatorname{inj}(y)\geq 4\epsilon$. Since $\iota$ is $1$-Lipschitz, we obtain $\ell_{g}(f^{-1}(t)\cap U_0)\geq 4\epsilon$ for any regular values $t\in (4\epsilon, 5\epsilon)$. Finally, we apply the coarea formula together with \eqref{ineq:volcomp}, \eqref{ineq:dfU} and obtain
	\begin{align*}
		A(\iota)\geq A(\iota|_{f^{-1}([4\epsilon,5\epsilon])\cap U_0})&>\operatorname{vol}_g(f^{-1}([4\epsilon,5\epsilon])\cap U_0)-\epsilon'\\
		&>\frac{1}{1+2\epsilon'}\int_{f^{-1}([4\epsilon,5\epsilon])\cap U_0}||df|| dV_g-\epsilon'\\
		&=\frac{1}{1+2\epsilon'}\int_{4\epsilon}^{5\epsilon}\ell_{g}(f^{-1}(t)\cap U_0)dt-\epsilon'\\
		&\geq \frac{1}{1+2\epsilon'} 4\epsilon^2-\epsilon'.
	\end{align*}
Since $\epsilon'>0$ is arbitrary, the lemma follows.

\end{proof}

Now we are ready to prove $(1)$ and $(2)$ of Theorem \ref{thm:bounded}.

\begin{theorem}\label{thm:collar}
Suppose that $\Gamma<\Isom(X)$ is a finitely generated, torsion-free, discrete isometry subgroup of a negatively pinched (normalized to $K\leq -1$) Hadamard manifold $X$. Let $N(\Ga)$ be the number of cusps in $M$, and $\beta_1(\Gamma)$ be the first Betti number of $M$. Suppose $\delta<1$, then we have the following.
	\begin{enumerate}
		\item $N(\Ga)\leq \beta_1(\Gamma)$.
		\item If an integer $k>\beta_1(\Gamma)-N(\Ga)$, then for any family of closed geodesics $\{\alpha_1,...,\alpha_k\}$ that are mutually $(2C(\delta)+1)$ apart, there exists at least one closed geodesic whose normal injectivity radius is $\leq C(\delta)$, where $C(\delta)=4/(1-\delta)$.
		\item $M$ has bounded geometry.
	\end{enumerate}
\end{theorem}

\begin{proof}
		For $(1)$, suppose to the contrary $N(\Ga)>\beta_1(\Gamma)$ ($N(\Ga)$ could be infinite). We choose $\epsilon$ small enough so that the cuspidal part $\textup{cusp}_{12\epsilon}(M)$ consists of $N(\Ga)$ disjoint components $\bigcup_{i=1}^N M_{12\epsilon}^{(i)}$. For each component $M_{12\epsilon}^{(i)}$, the corresponding parabolic subgroup $P_i$ is infinite cyclic by Proposition \ref{prop:parabolic}, so we can choose $\gamma_i\in P_i< \Gamma$ which represents a non-trivial torsion free homology class in $X/ P_{i}$ (not necessarily in $M$). Since $N(\Ga)>\beta_{1}(\Gamma)$, $\{[\gamma_1], \cdots [\gamma_{N(\Ga)}]\}$ is linearly dependent in $H_1(M)$. We can choose an irreducible subfamily containing $[\gamma_1]$ and without loss of generality we assume this to be $\{\gamma_1,...,\gamma_k\}$ where $k\leq \beta_{1}(\Gamma)+1<\infty$. Let $c_1,...,c_k$ be the associated integers such that $\sum_{i=1}^k c_i\cdot[\gamma_i]=0$ (with $c_1\neq 0$).
	On each component $M_{12\epsilon}^{(i)}$, we choose a thinner part $M_{4\epsilon}^{(i)}\subset M_{12\epsilon}^{(i)}$ and let $T_i=M_{12\epsilon}^{(i)}\backslash M_{4\epsilon}^{(i)}$. In particular, $T_i$ are disjoint and for any $x\in T_i$, we have $2\epsilon\leq\operatorname{inj}(x)\leq  6\epsilon$. We choose a loop $\alpha_i\subset M_{2\epsilon}^{(i)}$ representing $[\gamma_i]$ such that $\ell(\alpha_i)$ is so small that $\sum_{i=1}^k |c_i|\ell(\alpha_i)<\epsilon^2/C(\delta)$ \cite[Proposition 1.1.11]{bow3}. By Theorem \ref{thm:linear iso}, $\bigcup_{i=1}^k c_i \alpha_i$ bounds a smooth surface $\iota:\Sigma\rightarrow M$ whose area satisfies 
\begin{equation}
\label{eq:upperbound}
A(\iota)\leq C(\delta)\left(\sum_{i=1}^k |c_i| \ell(\alpha_i)\right)< \epsilon^2.
\end{equation}
However, by Lemma \ref{lem:lowerbound}, $A(\iota)\geq 4\epsilon^{2}$, which contradicts to \eqref{eq:upperbound}. Hence, $N(\Ga)\leq \beta_1(\Ga)$.

		For $(2)$, suppose there are $k=\beta_1(\Gamma)-N(\Ga)+1$ mutually $(2C(\delta)+1)$ apart simple closed geodesics $\alpha_1,...,\alpha_k$ whose normal injectivity radii are $>C(\delta)$. To illustrate the idea, we first assume $M$ has no cusps. Then $[\alpha_1],...,[\alpha_k]$ are linearly dependent on $H_1(M)$. By Theorem \ref{thm:linear iso}, there exist integers $c_1,...,c_k$ such that $\bigcup_{i=1}^k c_i \alpha_i$ bounds a smooth surface $f:\Sigma\rightarrow M$ whose area satisfies 
\begin{equation}
\label{eq:contradiction}	
	A(f)\leq C(\delta)\left(\sum_{i=1}^k |c_i| \ell(\alpha_i)\right).
\end{equation}
	Let $R_i=\operatorname{NJ}(\alpha_i)$ and by the assumption $R_i>C(\delta)$,  we can pick $\epsilon>0$ small enough so that $\epsilon<1/2$ and $C(\delta)+\epsilon<R_i$ for all $i$. Denote  $T_i$ the $(C(\delta)+\epsilon)$-tubular neighborhood of $\alpha_i$, and since $\{\alpha_i\}$ are mutually $(2C(\delta)+1)$ apart, $\{T_i\}$ are disjoint, and so are $\{f^{-1}(T_i)\}$. Therefore by Lemma \ref{lem:area-lower-bound}, we have
\begin{equation}\label{eq:area-lower-bound}
A(f)\geq \sum_{i=1}^k A(f|_{f^{-1}(T_i)})\geq (C(\delta)+\epsilon)\left(\sum_{i=1}^k |c_i|\ell(\alpha_i)\right).
\end{equation}
	This contradicts to \eqref{eq:contradiction}. 
	
For the general case, we pick non-trivial torsion free homology classes $\{[\gamma_1],...,[\gamma_{N(\Ga)}]\}$ on each cusp component   as in $(1)$. This together with $[\alpha_1],...,[\alpha_k]$ form a linearly dependent system on $H_1(M)$. We choose an irreducible system containing $[\alpha_1]$, and without loss of generality we assume it to be $\{[\gamma_1],...,[\gamma_{N(\Ga)}], [\alpha_1],...,[\alpha_k]\}$. Thus there are integers $b_1,...,b_{N(\Ga)}$ and $c_1,...,c_k$ such that $\sum_{i=1}^{N(\Ga)} b_i[\gamma_i]+\sum_{j=1}^k c_j[\alpha_j]=0$. Now we choose a loop $\eta_i$ on each cusp component representing $\gamma_i$ such that $\ell(\eta_i)$ is sufficiently small so that $\sum_{i=1}^{N(\Ga)}|b_i|\ell(\eta_i)<\epsilon\left(\sum_{j=1}^k|c_j| \ell(\alpha_j)\right)/C(\delta)$, where $\epsilon$ is the same constant as above in the non-cusp case. By Theorem \ref{thm:linear iso}, $\left(\bigcup_{i=1}^{N(\Ga)} b_i\eta_i\right)\cup\left(\bigcup_{j=1}^k c_j \alpha_j\right)$ bounds a smooth surface $f:\Sigma\rightarrow M$ whose area satisfies 
	\[A(f)\leq C(\delta)\left(\sum_{i=1}^{N(\Ga)} |b_i| \ell(\eta_i)+\sum_{j=1}^k|c_j| \ell(\alpha_j)\right).\]
Thus we have
\[A(f)< C(\delta)\left(1+\frac{\epsilon}{C(\delta)}\right)\left(\sum_{j=1}^k |c_j| \ell(\alpha_j)\right)=(C(\delta)+\epsilon)\left(\sum_{j=1}^k |c_j|\ell(\alpha_j)\right).
\]
However, the area lower bound estimate in \eqref{eq:area-lower-bound} still holds, and this gives a contradiction.

For $(3)$, suppose $M$ has unbounded geometry, that is, there exists a sequence of closed geodesics $\{\alpha_i\}$ with $\ell(\alpha_i)\rightarrow 0$. When $\ell(\alpha_i)$ is smaller than the Margulis constant, $\alpha_i$ determines a Margulis tube such that the length of every maximal radial arc  tends to $\infty$ as $\ell(\alpha_i)\rightarrow 0$. (See for example \cite[Lemma 2.4]{BCD}.) In particular, the normal injectivity radius $\operatorname{NJ}(\alpha_i)\rightarrow \infty$. By  passing to a subsequence, we can assume that the geodesics  $\alpha_i$ are arbitrarily far apart and their normal injectivity radii are all $>C(\delta)$, which  contradicts  to $(2)$.
\end{proof}

\begin{remark} The assumption $\delta<1$ is crucial in Theorem \ref{thm:collar} (which also traces back to Theorem \ref{thm:linear iso}). Indeed, the main strategy of the proof is to apply an area-decreasing map on the (approximated) area-minimizing surfaces which are bounded either by tiny loops in different cusps or by far apart closed geodesics. The existence of such map follows from a construction of Besson-Courtois-Gallot (Theorem \ref{thm:BCG}) where $\delta<1$ has been used to obtain the area decreasing.

In general, there are examples \cite{Kap1} of finitely generated Kleinian groups $\Gamma<\Isom(\mathbb H^4)$ with infinitely many (rank one) cusps, and by construction it is clear that $\delta\in [2,3]$. Thus, for every $n\geq 4$, one can construct, via the totally geodesic embedding $\mathbb H^4\rightarrow \mathbb H^n$, a Kleinian group $\Gamma< \Isom(\mathbb H^n)$ of the same critical exponent which contains infinitely many cusps. In a recent preprint \cite{IMM20}, Italiano-Martelli-Migliorini constructed new examples of finitely generated Kleinian groups $\Gamma\lhd G<\Isom(\mathbb H^n)$ ($5\leq n\leq 8$) with infinitely many cusps, where $G$ is a lattice and $G/\Gamma\cong \mathbb Z$. Hence it follows that $\delta(\Gamma)=\delta(G)=n-1$. We believe that finitely generated Kleinian groups must have finitely many cusps if $\delta<2$.
\end{remark}

We end this section by the following corollary which turns out to be essential to our proofs of the main theorems. It is a direct consequence of $(2)$ of Theorem \ref{thm:collar}. Roughly speaking, if $\delta<1$, then  closed geodesics asymptotically  have uniformly bounded tubular neighborhoods.

\begin{corollary}\label{prop:subseq-of-bdd-NJ}
Suppose $\delta<1$, and  $M$ has a sequence of escaping closed geodesics. Then there exists a subsequence of escaping closed geodesics whose normal injectivity radii are $\leq C(\delta)$.
\end{corollary}

\subsection{Decomposing a closed geodesic}\label{sec:cut-closed-geod}
Suppose $\alpha$ is a closed geodesic in $M$ with $\operatorname{NJ}(\alpha)\leq C(\delta)$. By definition, there exists $x_0\in M$ achieving the normal injectivity radius such that it projects to $\alpha$ in two different geodesic minimizing paths. The two geodesic paths have  an angle of $\pi$. Thus we can decompose $\alpha$ into two piecewise geodesic loops $\alpha'$ and $\alpha''$ as shown in Figure \ref{F3}. It is clear that their lengths satisfy $\ell(\alpha')+\ell(\alpha'')\leq \ell(\alpha)+4C(\delta)$.

\begin{figure}[H]
\centering
\includegraphics[width=4.0in]{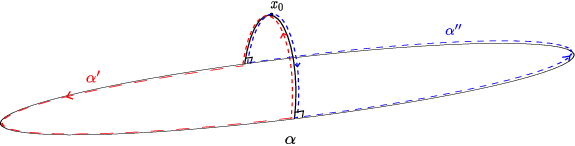}
\caption{ \label{F3}}
\end{figure}

Equivalently, in the universal cover (as shown in Figure \ref{F4}), there exists an isometry $g\in \Gamma$ and  $\tilde{x}_{0}\in X$ such that 
$$d(\tilde{x}_{0},A_{\ga})\leq C(\delta),\quad d(\tilde{x}_{0}, g^{-1}(A_{\ga}))\leq C(\delta),$$
where $A_{\ga}$ is a lift of $\alpha$ in $X$. Let $\tilde x, \tilde y$ be the projection of $\tilde{x}_{0}$ onto $g^{-1}(A_{\ga})$ and $A_{\ga}$ respectively, which will realize the shortest distance between $g^{-1}(A_{\ga})$ and $A_{\ga}$  (so $\ell(\tilde x \tilde y)\leq 2C(\delta)$). Under the projection map $\pi: X\rightarrow M$, the consecutive geodesic segments connecting $g(\tilde{x})$, $\tilde y$, $\tilde x$ maps to $\alpha'$ and the one connecting $\tilde x$, $\tilde y$, $\gamma\cdot g(\tilde x)$ maps to $\alpha''$, where $\gamma$ translates along $A_\gamma$ and corresponds to $\alpha$. From Figure \ref{F3}, we see that  $\alpha'$ represents the isometry $g$ and $\alpha''$ represents the isometry $\ga\cdot g$ which are  nontrivial elements in $\Gamma$. We claim that the group $\langle g, \ga\cdot g\rangle$ is nonelementary. Otherwise, $\langle g, \ga\cdot g\rangle$  is parabolic or loxodromic. If  $\langle g, \ga\cdot g \rangle$ is parabolic, then both $g$ and $\ga\cdot g$ are parabolic and they have the same fixed point,  which implies that $\gamma$ has the same fixed point as the one of the parabolic isometry $g$ and it  contradicts to the assumption that $\Ga$ is discrete by \cite[Lemma 3.1.2]{bow4}. (The proof of Lemma 3.1.2 can be applied to the case of negatively pinched Hadamard manifolds directly). 
 If $\langle g, \ga\cdot g\rangle$ is loxodromic, then $g$ and $\ga\cdot g$ are both loxodromic and they preserve  an axis setwise, which means that $\ga$ will preserve the same axis as $g$. However, note that $\ga$ preserves the axis $A_\ga$ which is not preserved by $g$.

\begin{figure}[H]
\centering
\includegraphics[width=3.0in]{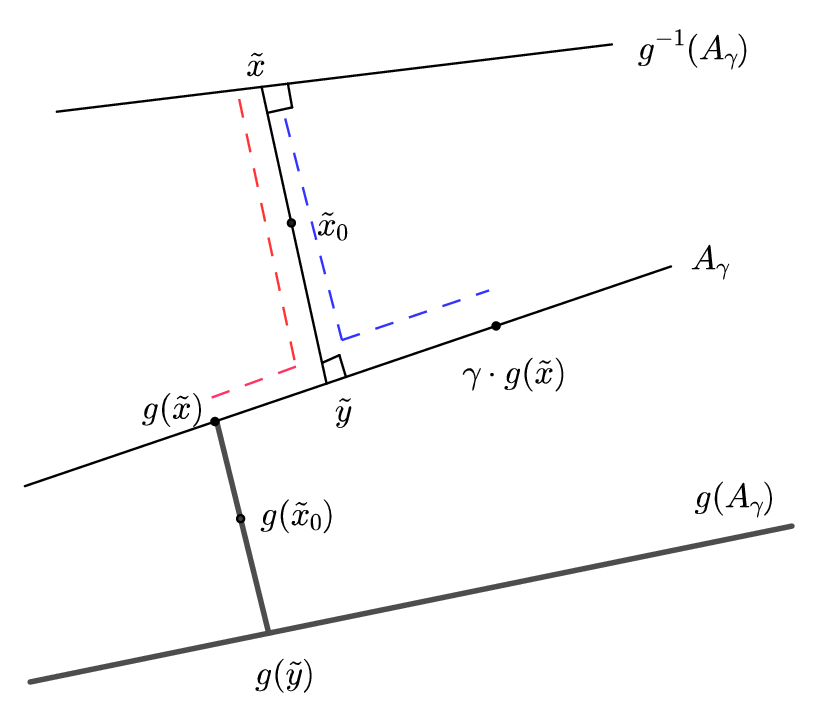}
\caption{ \label{F4}}
\end{figure}

It is possible that $x_0$ projects to the same point on $\alpha$, in which case $\alpha'$ is the entire transverse geodesic loop, and $\alpha''$ is the concatenation $\alpha'^{-1}\ast \alpha$. It is also possible that $\alpha$ may have a transverse self-intersection, in which case the above decomposition coincides with the obvious separation at the self-intersection. Note that non-transverse self-intersection of a closed geodesic $\alpha$ can only occur when $\alpha$ is a multiple of some primitive closed geodesic $\bar \alpha$, in which case  the above decomposition on $\alpha$ can essentially be treated on $\bar\alpha$. We remark that in all the above mentioned ``exceptional'' cases, the decomposition as described always exists. 

We can extend the above decomposition to a piecewise geodesic loop.

\begin{lemma}\label{lem:split-pw-geod}
Let $u\subset M$ be a piecewise geodesic  loop  consisting of at most two geodesics, and let $\alpha\subset M$ be the closed geodesic free homotopic to $u$ with $\operatorname{NJ}(\alpha)\leq C(\delta)$, and $\ell(\alpha)\geq \epsilon$.  Then there exist points $p, q\in u$ (which could be the same) and a geodesic segment $\omega$ connecting $p, q$ whose length is  bounded above by $C_0=2C(\delta)+2D(\epsilon)$. Here $D(\epsilon)$ is the constant in Proposition \ref{prop:piecegeod-close-to-geod}. Moreover, the two piecewise geodesic loops under the decomposition similar to Figure \ref{F3} are homotopically nontrivial. 
\end{lemma}

\begin{figure}[H]
\centering
\includegraphics[width=5.0in]{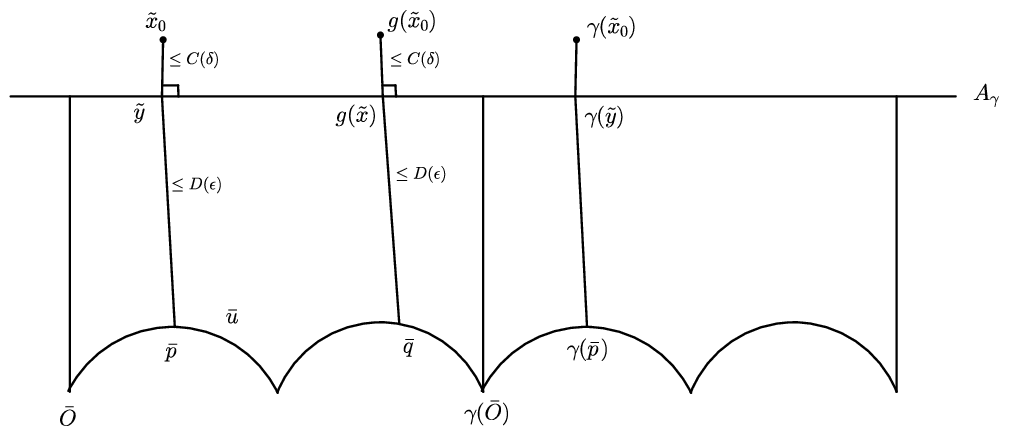}
\caption{ \label{F1}}

\end{figure}

\begin{proof} We write $u$ as the union of two geodesic segments in  $M$ which starts and ends at $O$. Let $\bar{u}$ be a lift of $u$ in $X$ consisting  of two geodesic segments from the lift $\bar{O}$ to $\gamma(\bar{O})$ as in Figure \ref{F1}, where $\gamma\in \Gamma$ is represented by $u$. We denote  the axis of $\gamma$ by $A_\gamma$ which is a lift of $\alpha$. Since $\operatorname{NJ}(\alpha)\leq C(\delta)$, by the discussion above there exists a point $\tilde{x}_{0}\in X$ and a nontrivial element $g\in \Gamma$ with $g\neq \gamma$, such that $\tilde{x}_{0}$ and $g(\tilde{x}_{0})$ project onto $A_\gamma$ at two points $\tilde{y}, g(\tilde{x})$ (which could be the same point) satisfying $d(\tilde{x}_{0}, \tilde{y})\leq C(\delta)$ and $d(g(\tilde{x}_{0}), g(\tilde{x}))\leq C(\delta)$, see Figure \ref{F4}.

By Proposition \ref{prop:piecegeod-close-to-geod}, there exist  $\bar p, \bar q\in \bar u$ such that $d(\tilde{y},\bar p)\leq D(\epsilon)$ and $d(g(\tilde{x}),\bar q)\leq D(\epsilon)$. Thus, the piecewise geodesic consecutively connecting $\bar p, \tilde{y}, \tilde{x}_{0}$ together with the one connecting $g(\tilde{x}_{0}), g(\tilde{x}), \bar q$ projects to   a piecewise geodesic path  connecting $\pi(\bar{p})=p,\pi(\bar{q})=q\in M$ whose total length is $\leq 2C(\delta)+2D(\epsilon)$. Finally, there is a unique geodesic segment $\omega$ connecting $p,q$ which is homotopic to  this piecewise geodesic path  and it is clear that $\ell(\omega)\leq 2C(\delta)+2D(\epsilon)$. 

The geodesic segment $\omega$ divides the piecewise geodesic loop $u$ into two parts $u_1$ and $u_2$. The concatenation of $u_i$ with the geodesic segment $\omega$ gives two piecewise geodesic loops under this decomposition where $i=1, 2$. If the two piecewise geodesic loops are  homotopically trivial, then $\tilde{x}_{0}=g(\tilde{x}_{0})=\ga(\tilde{x}_{0})$. By our construction,  $g\neq \gamma$ and $g\neq 1$. Hence, they are  homotopically nontrivial.

\end{proof}

\subsection{Injectivity radius and convex cocompactness}

In this section, we prove $(3)$ of Theorem \ref{thm:bounded}. We start by introducing the definition of a \emph{bow} which will be used later in the proof.

\begin{definition} Given a closed geodesic $\alpha$, we say $B=\overline{pq}\ast \wideparen{qp}$ is a \emph{bow} on $\alpha$ if
\begin{enumerate}
\item $B$ consists of $2$ edges $\overline{pq}$ and $\wideparen{qp}$ where $p, q$ are $2$ distinct points on $\alpha$,
\item $\overline{pq}$ is a minimizing geodesic connecting $p$ to $q$ on $M$, which might not lie on $\alpha$,
\item $\wideparen{qp}$ is a geodesic segment on $\alpha$ connecting $q$ to $p$, which might not be length minimizing, see Figure \ref{F5}.
\end{enumerate}
We say a bow $B=\overline{pq}\ast \wideparen{qp}$ is \emph{$C$-thin} if $d(p,q)\leq C$, and we say $B$ is \emph{non-trivial} if the loop $\overline{pq}\ast \wideparen{qp}$  of $B$ is homotopically non-trivial in $M$. The \emph{length} of a bow $B=\overline{pq}\ast \wideparen{qp}$ is the length of the loop  $\overline{pq}\ast \wideparen{qp}$.
\end{definition}

\begin{figure}[H]
\centering
\includegraphics[width=4in]{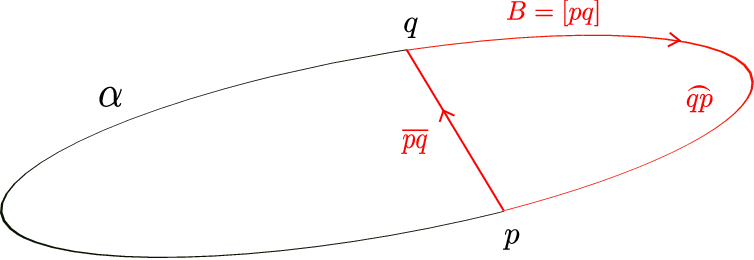}
\caption{ \label{F5}}

\end{figure}

\begin{lemma}
\label{lemma:nogeo}
Suppose that  $\delta<1$ and the injectivity radius on $M$ is bounded by some constant $\epsilon_0/2>0$ from below. Then there does not exist closed geodesic $\alpha$ in $M$ which satisfies the following two conditions:
\begin{enumerate}
\item $\alpha$ has normal injectivity radius at most $C(\delta)$;
\item all points of $\alpha$ have injectivity radii greater than $4C_0+1$, where $C_0$ is the constant in Lemma \ref{lem:split-pw-geod}. 
\end{enumerate}

\end{lemma}

\begin{proof}
Suppose that there exists such a closed geodesic $\alpha$ in $M$. We consider the set $\mathcal{B}=\mathcal B(\alpha, 2C_{0})$ that consists of all non-trivial $2C_{0}$-thin bows on $\alpha$. The set is never empty. Indeed, choose $p,q\in \alpha$ sufficiently close and choose $\wideparen{qp}$ the longer segment on $\alpha$ connecting $q$ to $p$, such that $\ell(\overline{pq})<\ell(\wideparen{qp})$ and $\ell(\overline{pq})\leq 2C_{0}$. This gives a nontrivial $2C_{0}$-thin bow on $\alpha$. 
Let $t=\inf\{\ell(B):B\in \mathcal B\}$. We choose $B=\overline{pq}\ast \wideparen{qp}\in \mathcal{B}$ to be a bow whose length $\leq t+1$. Since $B$ is a $2$-piecewise geodesic path, by Lemma \ref{lem:split-pw-geod}, there exist $r, s\in B$ and a geodesic segment $\omega\subset M$ connecting $r, s$ such that 
\begin{equation}
\label{for6}
\ell(\overline{rs})=\ell(\omega)\leq C_0
\end{equation}
and that $\omega$ splits $B$ non-trivially. Although Lemma \ref{lem:split-pw-geod} by itself does not assure $\omega$ is length minimizing, and $r,s$ might even be the same point, we claim this is not the case. Indeed, since $\ell(\overline{pq})\leq 2C_0$, $r$ must be contained in the $C_0$-neighborhood of $\alpha$. By the assumption on the injectivity radius,   all the  points on $\alpha$ have injectivity radius $>4C_{0}+1$.  Since the injectivity radius function  is $1$-Lipschitz, we have $\operatorname{inj}(r)> 3C_0+1$. This implies that any geodesic segment emanating from $r$ whose length is at most $3C_0+1$ must be uniquely length minimizing. In particular, $\omega$ is uniquely length minimizing and  $r\neq s$.

Based on the positions of $r$ and $s$, we now discuss in the following three cases separately:
\begin{enumerate}
\item $r, s$ are both on $\overline{pq}$,
\item $r, s$ are both on $\wideparen{qp}$,
\item $r\in \overline{pq}$ and $s\in \wideparen{qp}$.
\end{enumerate}

\begin{figure}[H]
\centering
\includegraphics[width=5.0in]{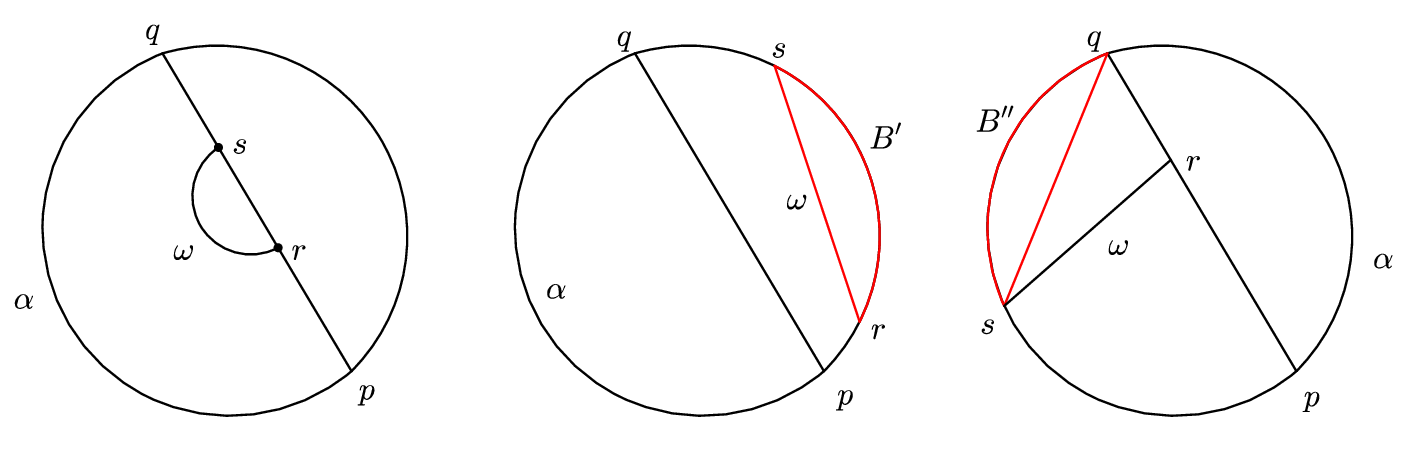}
\caption{ \label{F6}}

\end{figure}

Observe that $(1)$ is impossible since both  $\omega$ and $\overline{pq}$ are uniquely length minimizing, so $\omega$ has to be entirely contained in $\overline{pq}$, which contradicts to the fact that $\omega$ splits $B$ non-trivially. Case $(2)$ is also impossible. To see this, we assume without loss of generality that $q,s,r,p$ are in cyclic order in $\wideparen{qp}$ as in Figure \ref{F6}, and  $r,s$ cuts $\wideparen{qp}$ into three geodesic segments, denoted by $\wideparen{qs},\wideparen{sr},\wideparen{rp}$. By the assumption, the bow $B'=\overline{rs}\ast\wideparen{sr}$ is a non-trivial $C_0$-thin (of course also $2C_0$-thin) bow on $\alpha$. So by the choice of $B$ we have $\ell(B')+1\geq t+1\geq \ell(B)$ hence
\begin{equation}
\label{for5}
 \ell(\overline{rs})+1 \geq \ell(\wideparen{rp})+\ell(\overline{pq})+\ell(\wideparen{qs}).
\end{equation}
  Since $\omega$ splits $B$ non-trivially, we have obtained a homotopically non-trivial piecewise geodesic loop $\eta=\overline{rs}\ast \wideparen{sq}\ast\overline{qp}\ast \wideparen{pr}$ whose total length can be estimated as
\begin{align*}
\ell(\eta)&= \ell(\overline{rs}) +\ell(\wideparen{sq})+\ell(\overline{qp})+\ell(\wideparen{pr})\\
&\leq 2\ell(\overline{rs})+1  && \text{by } \eqref{for5}\\
&\leq 2C_0+1 && \text{by } \eqref{for6}.
\end{align*}
This contradicts to the  assumption on injectivity radius.

For Case $(3)$, note that  $\ell(\overline{pq})\leq 2C_0$, so $r$ is $C_0$ close to either $p$ or $q$, and without loss of generality we assume it is closer to $q$. Therefore by the  triangle inequality, we have $d(q,s)\leq \ell(\overline{rq})+\ell(\omega)\leq 2C_0$. Now we consider the bow $B''=\overline{sq}\ast \wideparen{qs}$ where $\wideparen{qs}$ is the geodesic segment on $\alpha$. The bow is nontrivial. Otherwise, $\overline{sq}$ coincides with $\wideparen{qs}$, which indicates that $\ell(\wideparen{qs})\leq 2C_{0}$. Then we have a piecewise geodesic loop $\overline{sr}\ast \overline{rq}\ast \wideparen{qs}$ with length $\leq 4C_{0}$. By the injectivity radius assumption, it must represent a trivial element, which contradicts to the fact that $\omega$ cuts $B_i$ nontrivially. Hence, $B''\in \mathcal{B}$. By the choice of $B$, we have $\ell(B'')+1\geq t+1\geq \ell(B)$ hence $\ell(\overline{sq})+1\geq \ell(\wideparen{sp})+\ell(\overline{pq})$. So we have obtained a piecewise geodesic loop $\eta'=\overline{qs}\ast \wideparen{sp}\ast\overline{pq}$ whose total length satisfies
\begin{align*}
\ell(\eta')&= \ell(\overline{qs})+\ell( \wideparen{sp})+\ell(\overline{pq})\\
&\leq 2\ell(\overline{qs})+1\\
&\leq 4C_0+1.
\end{align*}
So $\eta'$ must be homotopically trivial according to the injectivity radius assumption. Since $\omega$ splits $B_i$ non-trivially, the piecewise geodesic loop $\overline{rs}\ast \wideparen{sp}\ast \overline{pr}$ is homotopically non-trivial, and therefore differing by an $\eta'$, the geodesic triangle $\eta''=\overline{rs}\ast \overline{sq}\ast \overline{qr}$ is also homotopically non-trivial. On the other hand, we have
\[\ell(\eta'')=\ell(\overline{rs})+\ell( \overline{sq})+\ell(\overline{qr})\leq 4C_0,\]
which contradicts to the injectivity radius assumption.

\end{proof}

The following is a restatement of $(3)$ of Theorem \ref{thm:bounded}, which gives an alternative geometric characterization of convex compactness under the assumption that $\delta<1$.

\begin{theorem}\label{thm:cc-and-inj}
If $\delta<1$, then $\Gamma$ is convex cocompact if and only if the injectivity radius function $\textup{inj}: M\rightarrow \mathbb{R}$ is proper.
\end{theorem}

\begin{proof}
We start with the ``only if'' part which does not need the condition $\delta<1$. Since $\Gamma$ is convex cocompact, it consists of only loxodromic isometries. Note that all the closed geodesics are in the compact convex core since their lifts in $X$ are in $\Hull(\Lambda(\Gamma))$. Therefore, the length of all closed geodesics in $M$ is uniformly bounded from below. Otherwise, there is an escaping sequence of closed geodesics (whose length tends to $0$) inside the convex core, contradicting to the compactness. Suppose the injectivity radius function is not proper, then there exists an escaping sequence of points $x_i\in M$ whose injectivity radii are uniformly bounded by some constant $R$. At each point $x_i$, we choose a geodesic loop $w_i$ whose length satisfies $\ell(w_i)=2\operatorname{inj}(x_i)\leq 2R$. By Proposition \ref{prop:piecegeod-close-to-geod}, the closed geodesic free homotopic to $w_{i}$ is within $D$-neighborhood of $w_{i}$ for some constant $D$, hence we get an escaping sequence of closed geodesics in the convex core of $M$, and this contradicts to the compactness. 

To show the ``if'' part, we first note that properness of the injectivity radius function automatically implies that $M$ has no cusps, and  there is a uniform lower bound $\epsilon_0$ on the length of closed geodesics in $M$. Suppose that $\Gamma$ is not convex cocompact, i.e. geometrically infinite. By Theorem \ref{thm:KL} there is an escaping sequence of closed geodesics $\{\alpha_i\}\subset M$. By Corollary \ref{prop:subseq-of-bdd-NJ}, there is a subsequence of closed geodesics whose normal injectivity radii are all at most $C(\delta)$. For the convenience, we still denote it by $\{\alpha_i\}$. Now we fix a constant $C_0=2C(\delta)+2D(\epsilon_0)$ as in Lemma \ref{lem:split-pw-geod}. Since the injectivity radius function is proper and the sequence $\{\alpha_i\}$ is escaping, all points on $\alpha_i$ have injectivity radii greater than $4C_0+1$ when $i$ is sufficiently large. Hence, there exists a closed geodesic in $M$ whose normal injectivity radius is at most $C(\delta)$, and all points on the geodesic have injectivity radii greater than $4C_{0}+1$, contradicting to Lemma \ref{lemma:nogeo}. Therefore, $\Ga$ is convex cocompact. 

\end{proof}

\section{Proofs of the main  theorems}
\label{sec:proof}

\begin{theorem}\label{thm:small-implies-proper}
For each $n$ and $\kappa$, there exists a positive constant $D(n,\kappa)<1/2$ with the following property that, for any finitely generated, torsion-free discrete isometry subgroup $\Gamma<\Isom X$. If either
\begin{enumerate}
\item $\delta<D(n,\kappa)$, or
\item $\Gamma$ is free and $\delta<1/16$,
\end{enumerate}
then the injectivity radius function on $M$ is proper. 
\end{theorem}

\begin{proof}
Since $D(n, \kappa)<1/2$,  there are no parabolic isometries in $\Ga$ by Proposition \ref{prop:parabolic}. Suppose that the injectivity radius function is not proper. By the same argument as in the first paragraph of the proof of Theorem \ref{thm:cc-and-inj}, there exists an escaping sequence of closed geodesics  $\{\alpha_i\}$ of uniformly bounded length in $M$. 
Let $\mathcal G^{\infty}$ be the set of all escaping sequences of closed geodesics in $M$, and let $t=\inf\{\liminf_{i\rightarrow \infty}{\ell(\alpha_i)}:\{\alpha_i\}\in \mathcal G^{\infty}\}$. From the previous discussion, we see that $t<\infty$. On the other hand, $M$ has bounded geometry according to Theorem \ref{thm:collar}, so $t>0$. 

We claim that $t\leq 4C(\delta)$. Suppose $t>4C(\delta)$. Then  there exists an escaping sequence of closed geodesics $\alpha_i$ with $\liminf_{i\rightarrow \infty} \ell(\alpha_i)=s\in (t,t+\epsilon_0)$, where $\epsilon_{0}$ is a  fixed  positive number  smaller than $(t-4C(\delta))/2$. By Corollary \ref{prop:subseq-of-bdd-NJ} there exists a subsequence, which by abuse of notation we still denote by $\{\alpha_i\}$, such that $\lim_{i\rightarrow \infty}\ell(\alpha_i)=s$ and $\operatorname{NJ}(\alpha_i)\leq C(\delta)$ for all $i$. Without loss of generality, we assume $\ell(\alpha_i)\in (t,t+\epsilon_0)$ for all $i$. By Section \ref{sec:cut-closed-geod}, each $\alpha_i$ can be decomposed into two nontrivial loops $\alpha_i'$ and $\alpha_i''$ such that $\ell(\alpha_i')+\ell(\alpha_i'')\leq \ell(\alpha_i)+4C(\delta)$. So the shorter one, which we assume to be $\alpha_i'$, has length $\leq \frac{1}2 \ell(\alpha_{i})+2C(\delta)$, and it represents a nontrivial isometry in $\Gamma$. There is a  closed geodesic $\nu_i$ free homotopic to $\alpha_i'$ with length $\leq \frac{1}2 \ell(\alpha_{i})+2C(\delta)$. Since $M$ has bounded geometry, $\nu_{i}$ is inside a uniformly bounded neighborhood of $\alpha_i'$ by Proposition \ref{prop:piecegeod-close-to-geod}. Thus, we have found another escaping sequence of closed geodesics $\nu_i$, which satisfies
\begin{align*}
\ell(\nu_i)\leq \ell(\alpha'_i)&\leq \frac{1}{2}\ell(\alpha_i)+2C(\delta)\\
&\leq \frac{1}{2}(t+\epsilon_0)+2C(\delta)\\
&<\frac{1}{2}\left(t+\frac{1}{2}\left(t-4C(\delta)\right)\right)+2C(\delta)\\
&=\frac{3}{4}t+C(\delta).
\end{align*}
The last two inequalities follow from the choices of $\{\alpha_i\}$ and $\epsilon_{0}$.
Hence $\liminf_{i\rightarrow \infty}\ell(\nu_i)\leq \frac{3}{4}t+C(\delta)<t$.
This contradicts to the choice of $t$. Therefore, we have $t\leq 4C(\delta)$.

It means that for any $\epsilon>0$, there exists a primitive closed geodesic, denoted by $\alpha_0$, such that $\ell(\alpha_0)\leq t+\epsilon\leq 4C(\delta)+\epsilon$, and that $\operatorname{NJ}(\alpha_0)\leq C(\delta)$. By Section \ref{sec:cut-closed-geod}, $\alpha_0$ can be decomposed to two nontrivial loops $\alpha_0'$ and $\alpha_0''$, and again we assume $\alpha_0'$ is the shorter one. So $\ell(\alpha_0')<4C(\delta)+\epsilon$. Let $x_0$ be a common point of $\alpha_0$ and $\alpha_0'$.  Note that $\alpha_0$ and $\alpha_0'$ represent two loxodromic elements $\gamma_0,\gamma_0'\in \pi_1(M,x_0)\cong \Gamma$, which generate a non-elementary subgroup $\langle \gamma_{0}, \gamma'_{0}\rangle =\Gamma_0<\Gamma$.

Recall that for any group $G$ with finite generating set $S$, its entropy is defined as:
\[h(G,S)=\lim_{N\rightarrow \infty}\frac{\ln|\{g\in G:d_{S}(1,g)\leq N\}|}{N},\]
where $d_S$ is the Cayley  graph metric determined by $S$. 

If we are in case $(2)$ that $\Gamma$ is free, then $\Gamma_0$ must be a free subgroup isomorphic to $F_2$. So $h(\Gamma_0,S)=\ln 3$ for $S=\{\gamma_0,\gamma_0'\}$. Note that  the lengths of  geodesic loops from $x_0$ representing $\gamma_0$ and $\gamma_0'$ are both bounded  by $4C(\delta)+\epsilon$. We conclude that the orbit map $\gamma\mapsto \gamma\cdot x_0$ gives a $(4C(\delta)+\epsilon)$-Lipschitz injection from $(\Gamma_0,d_S)$ to $(X,d)$. This implies
\[\delta=\delta(\Gamma)\geq \delta(\Gamma_0)\geq \frac{1}{4C(\delta)+\epsilon} h(\Gamma_0,S)=\frac{\ln 3}{4C(\delta)+\epsilon}.\]
The last inequality follows from \eqref{for3}.
By choosing $\epsilon$ small enough and assuming $\delta<1/16$, one can check that the above inequality cannot hold. The contradiction implies that the injectivity radius is proper.

If we are in case $(1)$, then according to \cite[Theorem 1.1]{Dey-Kapovich-Liu}, there is a free subgroup $\Gamma_0'<\Gamma_0$ generated by two elements $g_0,g_0'$, whose word lengths measured in $(\Gamma_0,S)$ are bounded above by some universal constant $C(n,\kappa)$ depending only on the dimension and lower sectional curvature of $X$. Denote $S_0=\{g_0,g_0'\}$. Therefore, the orbit map $(\Gamma_0',d_{S_0})\rightarrow (X,d)$ through the inclusion $\Gamma_0'\rightarrow \Gamma_0$ is a $(4C(\delta)+\epsilon) C(n,\kappa)$-Lipschitz injection. This implies
\[\delta\geq \delta(\Gamma_0)\geq \frac{1}{(4C(\delta)+\epsilon)C(n,\kappa)} h(\Gamma_0',S_0)=\frac{\ln 3}{(4C(\delta)+\epsilon)C(n,\kappa)}.\]
Thus, there exists a constant $D(n,\kappa)$ which is smaller than $1/2$ such that by choosing $\epsilon$ small enough and assuming $\delta<D(n,\kappa)$, the above inequality fails. The contradiction again implies that the injectivity radius is proper.

\end{proof}

\begin{remark} For case (1), instead of passing to a rank 2 free subgroup, one can also apply the result of \cite{BCG2} to give a uniform lower bound on the entropy of $\Gamma_{0}$. 

\end{remark}

Now we can finish the proofs of our main results in the introduction.

\medskip
\noindent
{\bf Proofs of Theorem \ref{thm:main} and Theorem \ref{thm:mainfree}: }Theorem \ref{thm:main} follows from Theorem \ref{thm:cc-and-inj} and Theorem \ref{thm:small-implies-proper}. For the proof of Theorem \ref{thm:mainfree}, there exists a finite index free subgroup $\Ga'<\Ga$ such that $\delta(\Ga')= \delta(\Ga)<1/16$. Then $\Ga'$ is convex cocompact by Theorem \ref{thm:cc-and-inj} and Theorem \ref{thm:small-implies-proper}, which implies that $\Ga$ is also convex co-compact.  \qed

\medskip
\noindent
{\bf Proof of Corollary \ref{coro:virtuallyfree}: } Let $D(n)$ be the constant $D(n, \kappa)$ in Theorem \ref{thm:main} with $\kappa=1$. Suppose that $\Ga<\Isom(\H^{n})$ is a finitely generated discrete isometry subgroup with $\delta(\Ga)<D(n)<1/2$. By Selberg lemma, there exists a finite index torsion-free subgroup $\Ga'<\Ga$ with $\delta(\Ga')=\delta(\Ga)<D(n)<1/2$. By Theorem \ref{thm:main}, $\Ga'$ is convex cocompact. Hence, the Hausdorff dimension of the limit set equals $\delta(\Ga')$ \cite{BJ} which is smaller than 1. Note that since the limit set is a second countable, compact metric space (hence also locally compact and Hausdorff), its topological dimension equals the small inductive dimension, which is bounded above by its Hausdorff dimension, hence must be zero. This implies that the limit set is totally disconnected (in fact a Cantor set). Then we apply a result of Kulkarni \cite[Theorem 6.11]{Kulkarni78}, which states that if the limit set of a finitely generated Kleinian group is totally disconnected, then the group splits as a free amalgamation of a free group with virtually abelian groups corresponding to the parabolic subgroups. Since the condition $\delta(\Gamma')<1$ excludes all free abelian factors of higher rank, we conclude $\Ga'$ must be free. Therefore, $\Ga$ is virtually free. \qed


\end{document}